\documentclass[11pt, oneside, reqno]{amsart}
\usepackage[utf8]{inputenc}
\usepackage{mathrsfs}
\expandafter\newif\csname ifGin@setpagesize\endcsname 
\usepackage{ytableau}
\usepackage{bm}
\usepackage{amstext}
\usepackage{amsthm}
\usepackage{amssymb}
\usepackage{mathtools}
\usepackage[all]{xy}
\usepackage{cleveref}
\usepackage{graphics}

\numberwithin{equation}{section}
\numberwithin{figure}{section}
\theoremstyle{plain}
\newtheorem{thm}{Theorem}[section]
  \crefname{thm}{Theorem}{Theorems}
  \newtheorem{lem}[thm]{Lemma}
  \crefname{lem}{Lemma}{Lemmas}
  \newtheorem{prop}[thm]{Proposition}
  \crefname{prop}{Proposition}{Propositions}
  \newtheorem{cor}[thm]{Corollary}
	\crefname{cor}{Corollary}{Corollaries}
  \newtheorem*{theor1}{Theorem 1}
  \newtheorem*{theor2}{Theorem 2}
  \newtheorem*{theor3}{Theorem 3}
  \newtheorem*{cor4}{Corollary 4}
  \newtheorem*{ack}{Acknowledgments}
\theoremstyle{definition}
  \newtheorem{defi}[thm]{Definition}
  \crefname{defi}{Definition}{Definitions}
  \theoremstyle{remark}
  
  \crefname{ntn}{Notation}{Notations}
	 \theoremstyle{remark}
  \newtheorem{rem}[thm]{Remark}
  \crefname{rem}{Remark}{Remarks}
  \newtheorem{ex}[thm]{Example}
  \crefname{ex}{Example}{Examples}
  
\usepackage{a4wide}

\def\r{\mathbb{R}}
\def\c{\mathbb{C}}
\def\q{\mathbb{Q}}
\def\z{\mathbb{Z}}

\makeatother

\makeatletter
\ifcsname phantomsection\endcsname
    \newcommand*{\qrr@gobblenexttocentry}[5]{}
\else
    \newcommand*{\qrr@gobblenexttocentry}[4]{}
\fi
\newcommand*{\addsubsection}{%
    \addtocontents{toc}{\protect\qrr@gobblenexttocentry}%
    \subsection}
\makeatother
 \makeatletter
    
    \@addtoreset{equation}{section}
  \makeatother
\begin{document}
\title[Schubert calculus from polyhedral parametrizations of Demazure crystals]{Schubert calculus from polyhedral parametrizations of Demazure crystals}

\author{Naoki FUJITA}

\address[Naoki FUJITA]{Graduate School of Mathematical Sciences, The University of Tokyo, 3-8-1 Komaba, Meguro-ku, Tokyo 153-8914, Japan.}

\email{nfujita@ms.u-tokyo.ac.jp}

\subjclass[2010]{Primary 05E10; Secondary 14M15, 14M25, 14N15, 52B20}

\keywords{Schubert calculus, string polytope, Demazure crystal, reduced Kogan face, semi-toric degeneration}

\thanks{The work was partially supported by Grant-in-Aid for JSPS Fellows (No.\ 19J00123) and by Grant-in-Aid for Early-Career Scientists (No.\ 20K14281).}

\date{}

\begin{abstract}
One approach to Schubert calculus is to realize Schubert classes as concrete combinatorial objects such as Schubert polynomials.
Through an identification of the cohomology ring of the type A full flag variety with the polytope ring of the Gelfand--Tsetlin polytopes, Kiritchenko--Smirnov--Timorin realized each Schubert class as a sum of reduced (dual) Kogan faces.  
In this paper, we explicitly describe string parametrizations of opposite Demazure crystals, which give a natural generalization of reduced dual Kogan faces. 
We also relate reduced Kogan faces with Demazure crystals using the theory of mitosis operators, and apply these observations to develop the theory of Schubert calculus on symplectic Gelfand--Tsetlin polytopes. 
\end{abstract}
\maketitle
\ytableausetup{smalltableaux}
\tableofcontents 
\section{Introduction}\label{s:section}

\emph{Schubert calculus} was named after Hermann Schubert who developed a remarkable symbolic calculus in enumerative geometry of Grassmann varieties. 
One goal of Schubert calculus is to compute the structure constants of the cohomology ring of a (partial) flag variety with respect to the basis consisting of Schubert classes; see \cite{KST, KnM, Man} and references therein for the history of Schubert calculus.
One approach to such computation is to realize Schubert classes as concrete combinatorial objects such as Schubert polynomials. 
Through an identification of the cohomology ring of the type $A$ full flag variety with the polytope ring of the Gelfand--Tsetlin polytopes, Kiritchenko--Smirnov--Timorin \cite{KST} realized each Schubert class as a sum of reduced (dual) Kogan faces.  
In this paper, we generalize this result to the type $C$ full flag variety through the theory of Kashiwara crystal bases \cite{Kas2, Kas4}. 

To be more precise, let $G$ be a connected, simply-connected semisimple algebraic group over $\mathbb{C}$, $B \subseteq G$ a Borel subgroup, $G/B$ the full flag variety, and $W$ the Weyl group.
For $w \in W$, denote by $\ell(w)$ the length of $w$, by $R(w)$ the set of reduced words for $w$, and by $X_w \subseteq G/B$ (resp., $X^w \subseteq G/B$) the Schubert variety (resp., the opposite Schubert variety) with $\dim_\c(X_w) = \dim_\c(G/B) - \dim_\c (X^w) = \ell(w)$. 
Then the sets $\{[X_w] \mid w \in W\}$ and $\{[X^w] \mid w \in W\}$ of cohomology classes, called Schubert classes, form $\z$-bases of the cohomology ring of $G/B$:
\[H^\ast(G/B; \z) = \sum_{w \in W} \z [X_w] = \sum_{w \in W} \z [X^w].\]
If we write $[X^u] \cdot [X^v] = \sum_{w \in W} c_{u, v}^w [X^w]$ for some $c_{u, v}^w \in \z$, then we see by Kleiman's transversality theorem that $c_{u, v}^w \geq 0$ (see \cite[Sect.\ 1.3]{Bri}).
One goal of Schubert calculus is to compute this nonnegative integer $c_{u, v}^w$ explicitly. 
To proceed in such direction, the \emph{Borel description} is useful, which states that the (rational) cohomology ring $H^\ast(G/B; \q)$ is isomorphic to the coinvariant algebra of $W$ (see \cite{Man}).
This description allows us to represent each Schubert class as a polynomial. 
Such polynomials were studied by Bernstein--Gelfand--Gelfand \cite{BGG} using divided difference operators (see also \cite{Dem}).
In the case $G = SL_{n+1}(\c)$, Lascoux--Sch\"{u}tzenberger \cite{LS} gave a specific choice $\{\mathfrak{S}_w \mid w \in W\}$ of representatives, called Schubert polynomials, which have good combinatorial properties. 
Schubert polynomials for other classical groups were introduced by Billey--Haiman \cite{BH} and by Fomin--Kirillov \cite{FK_typeB}. 

Denote by $P_+$ the set of dominant integral weights, by $P_{++}$ the set of regular dominant integral weights, and by $V(\lambda)$ the irreducible highest weight $G$-module over $\c$ with highest weight $\lambda \in P_+$. 
This $G$-module $V(\lambda)$ has a remarkable combinatorial skeleton $\mathcal{B}(\lambda)$, called the Kashiwara crystal basis (see \cite{Kas2, Kas5}).
For $w \in W$ and $\lambda \in P_+$, there exists a subset $\mathcal{B}_w(\lambda) \subseteq \mathcal{B}(\lambda)$ (resp., $\mathcal{B}^w(\lambda) \subseteq \mathcal{B}(\lambda)$), called a Demazure crystal (resp., an opposite Demazure crystal), which naturally corresponds to $X_w$ (resp., $X^w$). 
Our approach to Schubert calculus is to use a polyhedral parametrization of an opposite Demazure crystal $\mathcal{B}^w(\lambda)$ as a combinatorial model of $[X^w]$. 
Let $w_0 \in W$ denote the longest element, and set $N \coloneqq \ell(w_0) = \dim_\c(G/B)$.
We associate to each ${\bm i} \in R(w_0)$ a parametrization $\Phi_{\bm i} \colon \mathcal{B}(\lambda) \hookrightarrow \z^N$ and a rational convex polytope $\Delta_{\bm i}(\lambda)$, called a string parametrization and a string polytope, respectively.
For $\lambda \in P_{++}$, Caldero \cite{Cal} constructed a flat degeneration of $G/B$ to the normal toric variety corresponding to $\Delta_{\bm i}(\lambda)$. 
This degeneration was geometrically interpreted by Kaveh \cite{Kav2} from the theory of Newton--Okounkov bodies. 
Using this degeneration, Morier-Genoud \cite{Mor} gave a semi-toric degeneration of a Richardson variety $X_v \cap X^w$ for $v, w \in W$ such that $X_v \cap X^w \neq \emptyset$.
More precisely, she showed that the image $\Phi_{\bm i}(\mathcal{B}_v(\lambda) \cap \mathcal{B}^w(\lambda))$ is given as the set of lattice points in a union of faces of $\Delta_{\bm i}(\lambda)$, and that the Richardson variety $X_v \cap X^w$ degenerates into the union of irreducible normal toric varieties corresponding to these faces.
For ${\bm i} \in R(w_0)$ and $w \in W$, we define $R({\bm i}, w)$ by \eqref{eq:reduced_word_compatible_subset} in Sect.\ \ref{s:string_opposite}.
Then the following theorem explicitly describes the faces which give the image $\Phi_{\bm i}(\mathcal{B}^w(\lambda))$; see Sect.\ \ref{s:string_opposite} for the definition of $F_{\bm k} (\Delta_{\bm i}(\lambda))$.

\begin{theor1}[{see \cref{c:string_parametrization_opposite_Demazure}}]
For ${\bm i} \in R(w_0)$, $w \in W$, and $\lambda \in P_+$, it holds that
\begin{align*}
&\Phi_{\bm i}(\mathcal{B}^w(\lambda)) = \bigcup_{{\bm k} \in R({\bm i}, w)} F_{\bm k} (\Delta_{\bm i}(\lambda)) \cap \z^N.
\end{align*}
\end{theor1}

Theorem 1 can also be seen as an explicit description of a semi-toric degeneration of $X^w$. 
Hence the set $R({\bm i}, w)$ gives a natural generalization of reduced pipe dreams as suggested in \cite[Introduction]{KoM}.
Such generalization was also given by Fomin--Kirillov \cite{FK_typeB} for type $BCD$ in the context of nilCoxeter algebras.
Kaveh \cite{Kav1} observed that the Borel description can be interpreted as an isomorphism between the cohomology ring $H^\ast(G/B; \q)$ and the polytope ring of $\{\Delta_{\bm i}(\lambda)\}_{\lambda \in P_+}$. 
If the family $\{\Delta_{\bm i}(\lambda)\}_{\lambda \in P_+}$ of string polytopes has a good combinatorial property, then the family $\{F_{\bm k} (\Delta_{\bm i}(\lambda))\}_{{\bm k} \in R({\bm i}, w)}$ gives the Schubert class $[X^w]$ under this isomorphism as we see below. 

Let us consider the case $G = SL_{n+1}(\c)$. 
We denote by $GT(\lambda)$ the Gelfand--Tsetlin polytope for $\lambda \in P_+$, and by $\mathfrak{S}_w({\bm x})$ the Schubert polynomial for $w \in W$.
Billey--Jockusch--Stanley \cite{BJS} and Fomin--Stanley \cite{FS} gave an explicit combinatorial formula 
\begin{equation}\label{eq:Schubert_polynomial_PD}
\begin{aligned}
\mathfrak{S}_w({\bm x}) = \sum_{D \in RP(w)} {\bm x}^D, 
\end{aligned}
\end{equation}
where we used the notation in \cite[Corollary 2.1.3]{KnM}. 
A diagrammatic interpretation of the index set $RP(w)$ was invented by Fomin--Kirillov \cite{FK} and developed by Bergeron--Billey \cite{BB} and by Knutson--Miller \cite{KnM, Mil}.
Under the diagrammatic interpretation, an element of $RP(w)$ is called an rc-graph or a reduced pipe dream.
By definition, each reduced pipe dream $D \in RP(w)$ corresponds to specific faces $F_D^\vee$ and $F_D$ of $GT(\lambda)$, called a reduced Kogan face and a reduced dual Kogan face, respectively, such that $w(F_D^\vee)^{-1} = w(F_D) = w$ in the notation of \cite[Sects.\ 3.3 and 4.3]{KST} (see \cite[Sect.\ 2.2.1]{Kog} and \cite[Sect.\ 4]{KoM}). 
Knutson--Miller \cite{KnM} gave a geometric proof of \eqref{eq:Schubert_polynomial_PD} using a flat degeneration of the (opposite) matrix Schubert variety $\overline{X}^{w_0 w w_0}$ whose limit has irreducible components naturally parametrized by $RP(w)$. 
Using this degeneration, Kogan--Miller \cite{KoM} proved that the opposite Schubert variety $X^{w_0 w w_0}$ degenerates into the union of irreducible normal toric varieties corresponding to the reduced dual Kogan faces $F_D$, $D \in RP(w)$.
Relations between Schubert classes and reduced (dual) Kogan faces were studied in \cite{Kog, Kir_IMRN, KST}. 
Using an isomorphism between $H^0(SL_{n+1}(\c)/B; \z)$ and the polytope ring $R_{GT}$ of the Gelfand--Tsetlin polytopes, Kiritchenko--Smirnov--Timorin \cite{KST} formulated the following equalities: 
\[[X^w] = \sum_{D \in RP(w^{-1})} [F_D^\vee] = \sum_{D \in RP(w_0 w w_0)} [F_D].\]
Note that there is a bijective correspondence between $RP(w^{-1})$ and $RP(w)$.
Even though $[X^w]$ is represented by $\mathfrak{S}_w({\bm x})$, the class $[F_D^\vee]$ does not correspond to ${\bm x}^D$ in \eqref{eq:Schubert_polynomial_PD}; the class $[F_D^\vee]$ is not necessarily an element of $R_{GT}$ as explained in \cite[Remark 2.5]{KST}. 
It is defined in a $\z$-module $M_{\widetilde{GT}, GT}$ extending $R_{GT}$, which is obtained through a resolution of $GT(\lambda)$ for $\lambda \in P_{++}$. 
We define ${\bm i}_A \in R(w_0)$ by \eqref{eq:reduced_word_type_A} in Sect.\ \ref{s:type_A}.
Then Littelmann \cite{Lit} gave a unimodular affine transformation from the string polytope $\Delta_{{\bm i}_A}(\lambda)$ to $GT(\lambda)$. 
Under this transformation, the family $\{F_{\bm k} (\Delta_{{\bm i}_A}(\lambda))\}_{{\bm k} \in R({\bm i}_A, w)}$ corresponds to the family $\{F_D\}_{D \in RP(w_0 w w_0)}$ of reduced dual Kogan faces. 
Hence Morier-Genoud's semi-toric degeneration \cite{Mor} can be seen as a generalization of Kogan--Miller's semi-toric degeneration \cite{KoM} by Theorem 1.
In order to relate reduced Kogan faces with Demazure crystals, we use the notions of ladder moves and mitosis operators introduced by Bergeron--Billey \cite{BB} and by Knutson--Miller \cite{KnM}, respectively. 
Bergeron--Billey \cite{BB} proved that $RP(w)$ is stable under ladder moves and is obtained from a specific element $D_{\rm bot}(w) \in RP(w)$ by applying sequences of ladder moves.
Knutson--Miller \cite{KnM, Mil} showed that $RP(w^{-1} w_0)$ is obtained from $RP(w_0)$ by applying a sequence of (transposed) mitosis operators ${\rm mitosis}_i^\top$ along a reduced word for $w^{-1}$. 
In this paper, we use another operator $M_i$ for $1 \leq i \leq n$, which is similar to ${\rm mitosis}_i^\top$ and is naturally related to the operator $\bigcup_{k \in \z_{\geq 0}} \tilde{f}_i^k$ on $\mathcal{B}(\infty)$ studied in \cite{Kas4}.
By applying to $RP(w_0)$ a sequence of operators $M_i$ along a specific reduced word for $w^{-1}$, we define a set $\mathscr{M}(w)$ of pipe dreams (see Sect.\ \ref{s:type_A} for the precise definition), which coincides with $RP(w^{-1} w_0)$ by properties of ladder moves and mitosis operators above.
For $D \in \mathscr{M}(w) = RP(w^{-1} w_0)$, let $F_{{\bm k}_D}^\vee (\Delta_{{\bm i}_A}(\lambda))$ denote the face of $\Delta_{{\bm i}_A}(\lambda)$ corresponding to $F_D^\vee$. 

\begin{theor2}[{see \cref{c:Kogan_face_string_polytope}}]
Let $G = SL_{n+1}(\c)$.
For $w \in W$ and $\lambda \in P_+$, it holds that 
\begin{align*}
&\Phi_{{\bm i}_A}(\mathcal{B}_w(\lambda)) = \bigcup_{D \in \mathscr{M}(w)} F_{{\bm k}_D}^\vee (\Delta_{{\bm i}_A}(\lambda)) \cap \z^N.
\end{align*}
\end{theor2}

Such relation between reduced Kogan faces and the character of $\mathcal{B}_w(\lambda)$ was previously studied in \cite{BF, KST, PS}. 
Theorem 2 gives their relation in the level of Demazure crystals.

Let us consider the case $G = Sp_{2n}(\c)$. 
We define ${\bm i}_C \in R(w_0)$ by \eqref{eq:reduced_word_type_C} in Sect.\ \ref{s:type_C}.
Then Littelmann \cite{Lit} gave a unimodular affine transformation from $\Delta_{{\bm i}_C}(\lambda)$ to the symplectic Gelfand--Tsetlin polytope $SGT(\lambda)$. 
The faces of $\Delta_{{\bm i}_C}(\lambda)$ containing the origin are naturally parametrized by skew pipe dreams (see Sect.\ \ref{s:type_C} and \cite{Kir}). 
In this paper, we generalize the notion of ladder moves to skew pipe dreams, and use it to define a set $\mathscr{M}(w)$ of skew pipe dreams for $w \in W$ as in the case of type $A_n$.
For $D \in \mathscr{M}(w)$, let $F_{{\bm k}_D}^\vee (\Delta_{{\bm i}_C}(\lambda))$ denote the corresponding face of $\Delta_{{\bm i}_C}(\lambda)$.

\begin{theor3}[{see \cref{t:main_result_string_Demazure_type_C}}]
Let $G = Sp_{2n}(\c)$.
For $w \in W$ and $\lambda \in P_+$, it holds that
\begin{align*}
&\Phi_{{\bm i}_C}(\mathcal{B}_w(\lambda)) = \bigcup_{D \in \mathscr{M}(w)} F_{{\bm k}_D}^\vee (\Delta_{{\bm i}_C}(\lambda)) \cap \z^N.
\end{align*}
\end{theor3}

Let $R_{SGT}$ denote the polytope ring of the symplectic Gelfand--Tsetlin polytopes.
By Theorems 1 and 3, we can generalize results in \cite{KST} to the case of type $C_n$ as follows.

\begin{cor4}[{see Corollaries \ref{c:description_of_Schubert_class_symplectic_Kogan} and \ref{c:description_of_Schubert_class_dual_symplectic_Kogan}}]
For each $w \in W$, the cohomology class $[X^w] \in H^\ast(Sp_{2n}(\c)/B; \z) \simeq R_{SGT}$ is described as follows:
\begin{align*}
[X^w] = \sum_{{\bm k} \in R({\bm i}_C, w)} [F_{\bm k} (\Delta_{{\bm i}_C}(\lambda))] = \sum_{D \in \mathscr{M}(w_0 w)} [F_{{\bm k}_D}^\vee (\Delta_{{\bm i}_C}(\lambda))].
\end{align*}
\end{cor4}

Such description of $[X^w] \in H^\ast(Sp_{2n}(\c)/B; \z)$ was previously given when $n = 2$, and there was a partial result when $n = 3$ (see \cite{Kir, KP}).

\begin{ack}\normalfont
The author is deeply grateful to Valentina Kiritchenko for useful explanations on models in \cite{KST}.
The author would also like to thank Hironori Oya for helpful comments and fruitful discussions.  
\end{ack}

\section{Kashiwara crystal bases}\label{s:crystal_bases}

In this section, we review some basic facts on crystal bases and string polytopes, following \cite{BZ, Kas2, Kas4, Kav1, Lit, NZ}. Let $G$ be a connected, simply-connected semisimple algebraic group over $\mathbb{C}$, $\mathfrak{g}$ its Lie algebra, and $H \subseteq G$ a maximal torus. Denote by $\mathfrak{h} \subseteq \mathfrak{g}$ the Lie algebra of $H$, by $\mathfrak{h}^\ast \coloneqq {\rm Hom}_\c (\mathfrak{h}, \c)$ its dual space, and by $\langle \cdot, \cdot \rangle \colon \mathfrak{h}^\ast \times \mathfrak{h} \rightarrow \c$ the canonical pairing. Let $P \subseteq \mathfrak{h}^\ast$ be the weight lattice, $\{\alpha_i \mid i \in I\} \subseteq P$ the set of simple roots, and $\{h_i \mid i \in I\} \subseteq \mathfrak{h}$ the set of simple coroots, where $I = \{1, \ldots, n\}$ is an index set for the vertices of the Dynkin diagram.

\begin{defi}[{\cite[Definition 1.2.1]{Kas4}}]
A \emph{crystal} is a set $\mathcal{B}$ equipped with maps 
\begin{itemize}
\item ${\rm wt} \colon \mathcal{B} \rightarrow P$, 
\item $\varepsilon_i \colon \mathcal{B} \rightarrow \z \cup \{-\infty\}$ and $\varphi_i \colon \mathcal{B} \rightarrow \z \cup \{-\infty\}$ for $i \in I$,
\item $\tilde{e}_i \colon \mathcal{B} \rightarrow \mathcal{B} \cup \{0\}$ and $\tilde{f}_i \colon \mathcal{B} \rightarrow \mathcal{B} \cup \{0\}$ for $i \in I$,
\end{itemize}
where $-\infty$ and $0$ are additional elements that are not contained in $\z$ and $\mathcal{B}$, respectively, such that the following holds: for $i \in I$ and $b \in \mathcal{B}$,
\begin{enumerate}
\item[(i)] $\varphi_i(b) = \varepsilon_i(b) + \langle{\rm wt}(b), h_i\rangle$,
\item[(ii)] ${\rm wt}(\tilde{e}_i b) = {\rm wt}(b) + \alpha_i$, $\varepsilon_i(\tilde{e}_i b) = \varepsilon_i(b) -1$, and $\varphi_i(\tilde{e}_i b) = \varphi_i(b) +1$ if $\tilde{e}_i b \in \mathcal{B}$,
\item[(iii)] ${\rm wt}(\tilde{f}_i b) = {\rm wt}(b) - \alpha_i$, $\varepsilon_i(\tilde{f}_i b) = \varepsilon_i(b) +1$, and $\varphi_i(\tilde{f}_i b) = \varphi_i(b) -1$ if $\tilde{f}_i b \in \mathcal{B}$,
\item[(iv)] $b^\prime = \tilde{e}_i b$ if and only if $b = \tilde{f}_i b^\prime$ for $b^\prime \in \mathcal{B}$,
\item[(v)] $\tilde{e}_i b = \tilde{f}_i b = 0$ if $\varphi_i(b) = -\infty$.
\end{enumerate}
\end{defi}

\begin{defi}[{see \cite[Sect.\ 1.2]{Kas4}}]
For crystals $\mathcal{B}_1$ and $\mathcal{B}_2$, a map $\psi \colon \mathcal{B}_1 \cup \{0\} \rightarrow \mathcal{B}_2 \cup \{0\}$ is called a \emph{strict morphism} of crystals from $\mathcal{B}_1$ to $\mathcal{B}_2$ if the following conditions hold:
\begin{enumerate}
\item[(i)] $\psi(0) = 0$,
\item[(ii)] ${\rm wt}(\psi(b)) = {\rm wt}(b)$, $\varepsilon_i(\psi(b)) = \varepsilon_i(b)$, and $\varphi_i(\psi(b)) = \varphi_i(b)$ for $i \in I$ and $b \in \mathcal{B}_1$ such that $\psi(b) \in \mathcal{B}_2$,
\item[(iii)] $\tilde{e}_i \psi(b) = \psi(\tilde{e}_i b)$ and $\tilde{f}_i \psi(b) = \psi(\tilde{f}_i b)$ for $i \in I$ and $b \in \mathcal{B}_1$,
\end{enumerate}
where we set $\tilde{e}_i \psi(b) = \tilde{f}_i \psi(b) = 0$ if $\psi(b) = 0$. If $\psi$ is injective in addition, then we call it a \emph{strict embedding} of crystals. 
\end{defi}

Let $B \subseteq G$ be a Borel subgroup containing $H$, $P_+ \subseteq P$ the set of dominant integral weights, and $c_{i, j} \coloneqq \langle \alpha_j, h_i \rangle$ the Cartan integer for $i, j \in I$.
The quotient space $G/B$ is called the \emph{full flag variety}, which is a nonsingular projective variety.  
Let $N_G(H)$ denote the normalizer of $H$ in $G$, and $W \coloneqq N_G(H)/H$ the Weyl group of $\mathfrak{g}$, which is generated by the set $\{s_i \mid i \in I\}$ of simple reflections. 
For $w \in W$, we denote by $\ell(w)$ the length of $w$, by $R(w)$ the set of reduced words for $w$, and by $\widetilde{w} \in N_G(H)$ a lift for $w$.
Let $w_0 \in W$ be the longest element, $B^- \coloneqq \widetilde{w_0} B \widetilde{w_0}^{-1}$ the Borel subgroup opposite to $B$, and $U^-$ the unipotent radical of $B^-$.
For $\lambda \in P_+$, denote by $V(\lambda)$ the irreducible highest weight $G$-module over $\c$ with highest weight $\lambda$ and with highest weight vector $v_{\lambda}$. 
For $\lambda \in P_+$, we define a line bundle $\mathcal{L}_\lambda$ on $G/B$ by
\[
\mathcal{L}_\lambda \coloneqq (G \times \mathbb{C})/B,
\] 
where the right $B$-action on $G \times \mathbb{C}$ is given by $(g, c) \cdot b \coloneqq (g b, \lambda(b) c)$ for $g \in G$, $c \in \mathbb{C}$, and $b \in B$. 
For a closed subvariety $Z \subseteq G/B$, the restriction of $\mathcal{L}_\lambda$ to $Z$ is also denoted by the same symbol $\mathcal{L}_\lambda$. 
By the Borel--Weil theorem, the space $H^0(G/B, \mathcal{L}_\lambda)$ of global sections is a $G$-module isomorphic to the dual module $V(\lambda)^\ast \coloneqq {\rm Hom}_\mathbb{C}(V(\lambda), \mathbb{C})$. 
Lusztig \cite{Lus_can, Lus_quivers, Lus1} and Kashiwara \cite{Kas1,Kas2,Kas3} constructed specific $\c$-bases of $\c[U^-]$ and $H^0(G/B, \mathcal{L}_\lambda)$ via the quantized enveloping algebra associated with $\mathfrak{g}$. 
These are called (the specialization at $q = 1$ of) the \emph{upper global bases} ($=$ the \emph{dual canonical bases}), and denoted by $\{G^{\rm up}(b) \mid b \in \mathcal{B}(\infty)\} \subseteq \c[U^-]$ and $\{G_{\lambda}^{\rm up}(b) \mid b \in \mathcal{B}(\lambda)\} \subseteq H^0(G/B, \mathcal{L}_\lambda)$, respectively. 
The index sets $\mathcal{B}(\infty)$ and $\mathcal{B}(\lambda)$ are typical examples of crystals, called the \emph{crystal bases}.
See \cite{Kas5} for a survey on crystal bases. 
We define an element $b_\infty \in \mathcal{B}(\infty)$ (resp., $b_\lambda \in \mathcal{B}(\lambda)$) by the condition that $G^{\rm up}(b_\infty) \in \c[U^-]$ is a constant function on $U^-$ (resp., $G_{\lambda}^{\rm up}(b_\lambda)$ is a lowest weight vector in $H^0(G/B, \mathcal{L}_\lambda)$).

\begin{thm}[{see \cite[Theorem 5]{Kas2}}]\label{p:relations_crystals}
For $\lambda \in P_+$, there exists a unique surjective map 
\[\pi_\lambda \colon \mathcal{B} (\infty) \twoheadrightarrow \mathcal{B} (\lambda) \cup \{0\}\]
such that $\pi_\lambda (b_\infty) = b_\lambda$ and $\tilde{f}_i \pi_\lambda(b) = \pi_\lambda (\tilde{f}_i b)$ for all $i \in I$ and $b \in \mathcal{B}(\infty)$.
In addition, for 
\[\widetilde{\mathcal{B}} (\lambda) \coloneqq \{b \in \mathcal{B}(\infty) \mid \pi_\lambda(b) \neq 0\},\] 
the restriction map $\pi_\lambda \colon \widetilde{\mathcal{B}} (\lambda) \rightarrow \mathcal{B} (\lambda)$ is bijective. 
\end{thm}

For $w \in W$, we set 
\[X_w \coloneqq \overline{B \widetilde{w} B/B} \subseteq G/B\quad ({\rm resp.}, X^w \coloneqq \overline{B^- \widetilde{w} B/B} \subseteq G/B),\]
called the \emph{Schubert variety} (resp., the \emph{opposite Schubert variety}) associated with $w$. 
These varieties $X_w$ and $X^w$ are normal projective varieties, and we have 
\[\widetilde{w_0} X_w = \overline{\widetilde{w_0} B \widetilde{w} B/B} = \overline{B^- \widetilde{w_0} \widetilde{w} B/B} = X^{w_0 w}.\]
Let $\leq$ denote the Bruhat order on $W$.
For $v, w \in W$ such that $v \leq w$, the scheme-theoretic intersection 
\[X_w^v \coloneqq X_w \cap X^v \subseteq G/B\]
is reduced and irreducible (see, for instance, \cite[Proposition 1.3.2]{Bri}), which is called a \emph{Richardson variety}. 
Note that both $X_{w_0}$ and $X^e$ coincide with the full flag variety $G/B$, where $e \in W$ is the identity element.
Hence we have $X_{w_0}^w = X^w$ and $X_w^e = X_w$ for all $w \in W$.
In addition, it holds that $\dim_\c (X_w^v) = \ell(w) - \ell(v)$; see, for instance, \cite[Sect.\ 1.3]{Bri}. 
For $w \in W$ and $\lambda \in P_+$, we define a $B$-submodule $V_w(\lambda) \subseteq V(\lambda)$ (resp., a $B^-$-submodule $V^w(\lambda) \subseteq V(\lambda)$) as follows:
\begin{align*}
V_w(\lambda) \coloneqq \sum_{b \in B} \c b \widetilde{w} v_{\lambda}\quad({\rm resp.},\ V^w(\lambda) \coloneqq \sum_{b \in B^-} \c b \widetilde{w} v_{\lambda}),
\end{align*}
which is called a \emph{Demazure module} (resp., an \emph{opposite Demazure module}).  
Then we have
\begin{equation}\label{eq:longest_Demazure}
\begin{aligned}
\widetilde{w_0} V_w(\lambda) = \sum_{b \in B} \c \widetilde{w_0} b \widetilde{w_0}^{-1} \widetilde{w_0} \widetilde{w} v_{\lambda} = V^{w_0 w}(\lambda).
\end{aligned}
\end{equation}
By the Borel--Weil type theorem (see, for instance, \cite[Corollary 8.1.26]{Kum}), the space of global sections $H^0(X_w, \mathcal{L}_\lambda)$ is a $B$-module isomorphic to the dual module $V_w (\lambda)^\ast \coloneqq {\rm Hom}_\mathbb{C}(V_w(\lambda), \mathbb{C})$.
From this, we deduce that 
\begin{align*}
H^0(X^w, \mathcal{L}_\lambda) = H^0(\widetilde{w_0} X_{w_0 w}, \mathcal{L}_\lambda) \simeq (\widetilde{w_0} V_{w_0 w}(\lambda))^\ast = V^w (\lambda)^\ast
\end{align*}
as $B^-$-modules. 
Let $\pi_w \colon H^0(G/B, \mathcal{L}_\lambda) \rightarrow H^0(X_w, \mathcal{L}_\lambda)$ (resp., $\pi^w \colon H^0(G/B, \mathcal{L}_\lambda) \rightarrow H^0(X^w, \mathcal{L}_\lambda)$) denote the restriction map.

\begin{prop}[{\cite[Proposition 3.2.3 (i) and equation (4.1)]{Kas4}}] 
For $w \in W$ and $\lambda \in P_+$, there uniquely exists a subset $\mathcal{B}_w(\lambda)$ of $\mathcal{B}(\lambda)$ such that $\{\pi_w (G^{\rm up} _\lambda(b)) \mid b \in \mathcal{B}_w(\lambda)\}$ forms a $\c$-basis of $H^0(X_w, \mathcal{L}_\lambda)$ and such that $\pi_w (G^{\rm up} _\lambda(b)) = 0$ for all $b \in \mathcal{B}(\lambda) \setminus \mathcal{B}_w(\lambda)$.
Similarly, there uniquely exists a subset $\mathcal{B}^w(\lambda)$ of $\mathcal{B}(\lambda)$ such that it satisfies analogous conditions for $\pi^w$.
\end{prop}

These subsets $\mathcal{B}_w(\lambda)$ and $\mathcal{B}^w(\lambda)$ are called a \emph{Demazure crystal} and an \emph{opposite Demazure crystal}, respectively. 

\begin{prop}[{\cite[Propositions 3.2.3 (ii), (iii) and 4.2]{Kas4}}]\label{p:properties of Demazure}
Let $w \in W$, and $\lambda \in P_+$.
\begin{enumerate}
\item[{\rm (1)}] It holds that $\tilde{e}_i \mathcal{B}_w(\lambda) \subseteq \mathcal{B}_w(\lambda) \cup \{0\}$ and $\tilde{f}_i \mathcal{B}^w(\lambda) \subseteq \mathcal{B}^w(\lambda) \cup \{0\}$ for all $i \in I$.  
\item[{\rm (2)}] If $s_i w < w$, then it holds that 
\begin{align*}
\mathcal{B}_w(\lambda) = \bigcup_{k \in \z_{\geq 0}} \tilde{f}_i ^k \mathcal{B}_{s_i w}(\lambda) \setminus \{0\}\quad {\it and}\quad \mathcal{B}^{s_i w}(\lambda) = \bigcup_{k \in \z_{\geq 0}} \tilde{e}_i ^k \mathcal{B}^w(\lambda) \setminus \{0\}.
\end{align*}
\end{enumerate}
\end{prop}

For $\lambda \in P_+$, we have $-w_0 \lambda \in P_+$, and the crystal $\mathcal{B}(-w_0 \lambda)$ is identified with the dual crystal of $\mathcal{B}(\lambda)$ (see \cite[Example 1.2.10]{Kas4}). 
Under this identification, an opposite Demazure crystal $\mathcal{B}^w(\lambda)$ corresponds to the Demazure crystal $\mathcal{B}_{w w_0}(-w_0 \lambda)$. 
For $i \in I$, an \emph{$i$-string} $L$ is a subset of $\mathcal{B}(\lambda)$ with $b_L ^{\rm low} \in L$ such that $\tilde{f}_i b_L ^{\rm low} = 0$ and such that $L = \{\tilde{e}_i ^k b_L ^{\rm low} \mid k \in \z_{\geq 0}\} \setminus \{0\}$.
The following is called the \emph{string property} of opposite Demazure crystals.

\begin{prop}[{see \cite[Proposition 3.3.5]{Kas4}}]\label{p:string property}
Let $w \in W$, $\lambda \in P_+$, and $L$ an $i$-string of $\mathcal{B}(\lambda)$ for $i \in I$.
Then the intersection $\mathcal{B}^w(\lambda) \cap L$ is either $\emptyset$, $L$, or $\{b_L ^{\rm low}\}$.
\end{prop}

The action of $\widetilde{w_0}$ on $V(\lambda)$ induces a bijective involution $\eta_\lambda \colon \mathcal{B}(\lambda) \rightarrow \mathcal{B}(\lambda)$ (see, for instance, \cite[Sect.\ 2.1]{Mor}). 
By \eqref{eq:longest_Demazure}, it follows that 
\begin{equation}\label{eq:involution_longest element_for_Demazure_crystal}
\begin{aligned}
\eta_\lambda (\mathcal{B}_w(\lambda)) = \mathcal{B}^{w_0 w}(\lambda) & & \text{and} & & \eta_\lambda (\mathcal{B}^w(\lambda)) = \mathcal{B}_{w_0 w}(\lambda)
\end{aligned}
\end{equation}
for $w \in W$.
The natural projection $G \twoheadrightarrow G/B$ induces an open embedding $U^- \hookrightarrow G/B$ by which we identify $U^-$ with an affine open subvariety of $G/B$. 
Then the intersection $U^- \cap X_w$ in $G/B$ is regarded as a closed subvariety of $U^-$.
For $b \in \mathcal{B}(\infty)$, let $G_w^{\rm up}(b) \in \c[U^- \cap X_w]$ denote the restriction of $G^{\rm up}(b) \in \c[U^-]$. 

\begin{prop}[{see \cite[Proposition 3.2.5]{Kas4} (and also \cite[Corollary 3.20]{FO})}]\label{p:property_of_Demazure_infty}
For $w \in W$, there uniquely exists a subset $\mathcal{B}_w(\infty) \subseteq \mathcal{B}(\infty)$, called a \emph{Demazure crystal}, such that the following hold: 
\begin{itemize}
\item it holds that $\tilde{e}_i \mathcal{B}_w(\infty) \subseteq \mathcal{B}_w(\infty) \cup \{0\}$ for all $i \in I$;
\item if $s_i w < w$, then it holds that 
\begin{align*}
\mathcal{B}_w(\infty) = \bigcup_{k \in \z_{\geq 0}} \tilde{f}_i ^k \mathcal{B}_{s_i w}(\infty);
\end{align*}
\item the set $\{G_w^{\rm up}(b) \mid b \in \mathcal{B}_w(\infty)\}$ forms a $\c$-basis of $\c[U^- \cap X_w]$;
\item the equality $G_w^{\rm up}(b) = 0$ holds for all $b \in \mathcal{B}(\infty) \setminus \mathcal{B}_w (\infty)$.
\end{itemize}
\end{prop}

Writing $\widetilde{\mathcal{B}}_w(\lambda) \coloneqq \mathcal{B}_w(\infty) \cap \widetilde{\mathcal{B}}(\lambda)$ for $w \in W$ and $\lambda \in P_+$, it follows that
\begin{equation}\label{eq:Demazure_infty_limit}
\begin{aligned}
\mathcal{B}_w(\infty) = \bigcup_{\lambda \in P_+} \widetilde{\mathcal{B}}_w(\lambda)
\end{aligned}
\end{equation}
by \cite[Corollary 4.4.5]{Kas2}. 
We set $N \coloneqq \ell(w_0) = \dim_\c (G/B)$.

\begin{defi}[{see \cite[Sect.\ 1]{Lit}}]\label{d:string_parametrization}
Let ${\bm i} = (i_1, \ldots, i_N) \in R(w_0)$, and $\lambda \in P_+$.
For $b \in \mathcal{B}(\lambda)$ (resp., $b \in \mathcal{B}(\infty)$), we define $\Phi_{\bm i} (b) = (a_1, \ldots, a_N) \in \z_{\geq 0}^N$ by 
\[a_k \coloneqq \max\{a \in \z_{\geq 0} \mid \tilde{e}_{i_k} ^a \tilde{e}_{i_{k-1}} ^{a_{k-1}} \cdots \tilde{e}_{i_1} ^{a_1} b \neq 0\}\quad {\rm for}\ 1 \leq k \leq N.\]
It is called \emph{Berenstein--Littelmann--Zelevinsky's string parametrization} associated with ${\bm i}$. 
\end{defi}

The maps $\Phi_{\bm i} \colon \mathcal{B}(\lambda) \rightarrow \z_{\geq 0} ^N$ and $\Phi_{\bm i} \colon \mathcal{B}(\infty) \rightarrow \z_{\geq 0} ^N$ are indeed injective. 
Let $\mathcal{C}_{\bm i} \subseteq \r^N$ denote the smallest real closed cone containing $\Phi_{\bm i} (\mathcal{B}(\infty))$, which is called a \emph{string cone}.

\begin{prop}[{see \cite[Sect.~3.2]{BZ} and \cite[Sect.~1]{Lit}}]\label{p:string_cone_property}
For ${\bm i} \in R(w_0)$, the string cone $\mathcal{C}_{\bm i}$ is a rational convex polyhedral cone, and the equality $\Phi_{\bm i} (\mathcal{B}(\infty)) = \mathcal{C}_{\bm i} \cap \z^N$ holds.
\end{prop} 

\begin{defi}[{see \cite[Definition 3.5]{Kav2} and \cite[Sect.\ 1]{Lit}}]\label{d:string_polytopes}
For ${\bm i} \in R(w_0)$ and $\lambda \in P_+$, we define a subset $\mathcal{S}_{\bm i} (\lambda) \subseteq \z_{>0} \times \z^N$ by 
\[\mathcal{S}_{\bm i} (\lambda) \coloneqq \bigcup_{k \in \z_{>0}} \{(k, \Phi_{\bm i}(b)) \mid b \in \mathcal{B}(k\lambda)\}.\] 
Denote by $\mathcal{C}_{\bm i} (\lambda) \subseteq \mathbb{R}_{\geq 0} \times \mathbb{R}^N$ the smallest real closed cone containing $\mathcal{S}_{\bm i} (\lambda)$. 
Then we define a subset $\Delta_{\bm i} (\lambda) \subseteq \mathbb{R}^N$ by 
\[\Delta_{\bm i} (\lambda) \coloneqq \{{\bm a} \in \mathbb{R}^N \mid (1, {\bm a}) \in \mathcal{C}_{\bm i} (\lambda)\},\] 
which is called \emph{Berenstein--Littelmann--Zelevinsky's string polytope}.
\end{defi}

\begin{prop}[{see \cite[Sect.\ 3.2]{BZ} and \cite[Sect.\ 1]{Lit}}]\label{string lattice points}
For ${\bm i} \in R(w_0)$ and $\lambda \in P_+$, the string polytope $\Delta_{\bm i} (\lambda)$ is a rational convex polytope, and the equality $\Delta_{\bm i} (\lambda) \cap \mathbb{Z}^N = \Phi_{\bm i} (\mathcal{B}(\lambda))$ holds.
\end{prop}

The ring $\r[P_\r]$ of polynomial functions on $P_\r \coloneqq P \otimes_\z \r$ with $\r$-coefficients is naturally identified with the symmetric algebra ${\rm Sym}(P_\r^\ast)$ of $P_\r^\ast \coloneqq {\rm Hom}_\r(P_\r, \r)$. 
For $\lambda \in P_\r$, we define a derivation $\partial_\lambda$ on ${\rm Sym}(P_\r^\ast) \simeq \r[P_\r]$ by the Leibniz rule and by $\partial_\lambda (f) = f(\lambda)$ for $f \in P_\r^\ast$. 
Then the symmetric algebra ${\rm Sym}(P_\r)$ of $P_\r$ can be regarded as the ring of differential operators on $\r[P_\r]$ with $\r$-coefficients. 
The string polytopes $\Delta_{\bm i} (\lambda)$, $\lambda \in P_+$, inherit information of the cohomology ring of $G/B$ as follows. 

\begin{thm}[{\cite[Corollary 5.3]{Kav1}}]\label{t:Borel_description_polytope}
The cohomology ring $H^\ast(G/B; \r)$ of $G/B$ over $\r$ is isomorphic to the quotient ${\rm Sym}(P_\r)/J$ of ${\rm Sym}(P_\r)$ by the homogeneous ideal $J$ defined to be 
\[J \coloneqq \{D \in {\rm Sym}(P_\r) \mid D \cdot P_{G/B} = 0\},\] 
where $P_{G/B} \in \r[P_\r]$ denotes the homogeneous polynomial of degree $n\ (= {\it the\ rank\ of}\ G)$ on $P_\r$ determined by 
\[P_{G/B} (\lambda) = {\rm Vol}_N (\Delta_{\bm i} (\lambda))\quad {\it for\ all}\ \lambda \in P_+.\]
\end{thm}

Let $\Phi^+$ denote the set of positive roots, and set $\rho \coloneqq \frac{1}{2}\sum_{\alpha \in \Phi^+} \alpha \in P_+$.
By the Weyl dimension formula, we see that 
\begin{equation}\label{eq:Weyl_volume_formula}
\begin{aligned}
{\rm Vol}_N (\Delta_{\bm i} (\lambda)) = \prod_{\alpha \in \Phi^+} \frac{(\lambda, \alpha)}{(\rho, \alpha)}
\end{aligned}
\end{equation}
for $\lambda \in P_+$, where $(\cdot, \cdot)$ is a $W$-invariant inner product on $P_\r$.
Hence the description of $H^\ast(G/B; \r)$ in \cref{t:Borel_description_polytope} coincides with the Borel description (see \cite[Remark 5.4]{Kav1} for more details). 

\begin{rem}[{see, for instance, \cite{Man}}]\label{r:Borel_description_over_Z}
The Borel description holds for the cohomology ring $H^\ast(G/B; \z)$ over $\z$ when $G = SL_{n+1}(\c)$ or $G = Sp_{2n}(\c)$. 
Hence \cref{t:Borel_description_polytope} also holds for $H^\ast(G/B; \z)$ in these cases. 
\end{rem}

The cohomology ring $H^\ast(G/B; \z)$ is isomorphic to the Chow ring $A^\ast(G/B)$ of $G/B$ (see \cite[Example 19.1.11]{Ful}), and it is easy to see that $[X^{w_0 w}] = [\widetilde{w_0} X_w] = [X_w]$ in $A^\ast(G/B)$ for all $w \in W$.

For ${\bm i} = (i_1, \ldots, i_N) \in R(w_0)$, consider an infinite sequence ${\bm j} = (\ldots, j_k, \ldots, j_{N+1}, j_N, \ldots, j_1)$ of elements in $I$ such that $j_k = i_k$ for $1 \le k \le N$, $j_k \neq j_{k+1}$ for all $k \in \z_{\geq 1}$, and there exist infinitely many $k \in \z_{\geq 1}$ with $j_k = i$ for each $i \in I$. 
Following \cite{Kas4, NZ}, a crystal structure on 
\[\z^{\infty} \coloneqq \{(\ldots, a_k, \ldots, a_2, a_1) \mid a_k \in \z\ {\rm for}\ k \in \z_{\geq 1}\ {\rm and}\ a_k = 0\ {\rm for}\ k \gg 0\}\] 
is associated to ${\bm j}$ as follows. 
For $k \in \z_{\geq 1}$, $i \in I$, and ${\bm a} = (\ldots, a_l, \ldots, a_2, a_1) \in \z^\infty$, set 
\begin{align*}
&\sigma_k({\bm a}) \coloneqq a_k + \sum_{l > k} c_{j_k, j_l} a_l \in \z,\\
&\sigma^{(i)}({\bm a}) \coloneqq \max\{\sigma_k({\bm a}) \mid k \in \z_{\geq 1},\ j_k = i\} \in \z,\ {\rm and}\\
&M^{(i)}({\bm a}) \coloneqq \{k \in \z_{\geq 1} \mid j_k = i,\ \sigma_k({\bm a}) = \sigma^{(i)}({\bm a})\}.
\end{align*}
Since we have $a_l = 0$ for $l \gg 0$, the integers $\sigma_k({\bm a}), \sigma^{(i)}({\bm a})$ are well-defined, and it follows that $\sigma^{(i)}({\bm a}) \geq 0$. 
In addition, the cardinality of $M^{(i)}({\bm a})$ is finite if and only if $\sigma^{(i)}({\bm a}) > 0$. 
Let us define a crystal structure on $\z^\infty$ by 
\begin{align*}
&{\rm wt}({\bm a}) \coloneqq - \sum_{k = 1} ^\infty a_k \alpha_{j_k},\ \varepsilon_i({\bm a}) \coloneqq \sigma^{(i)}({\bm a}),\ \varphi_i({\bm a}) \coloneqq \varepsilon_i ({\bm a}) + \langle {\rm wt}({\bm a}), h_i \rangle,\\
&\tilde{e}_i {\bm a} \coloneqq 
\begin{cases}
(a_k - \delta_{k, \max M^{(i)}({\bm a})})_{k \in \z_{\geq 1}} &{\rm if}\ \sigma^{(i)}({\bm a}) > 0,\\
0 &{\rm otherwise},
\end{cases}\\
&\tilde{f}_i {\bm a} \coloneqq (a_k + \delta_{k, \min M^{(i)}({\bm a})})_{k \in \z_{\geq 1}} 
\end{align*}
for $i \in I$ and ${\bm a} = (\ldots, a_k, \ldots, a_2, a_1) \in \z^\infty$, where $\delta_{k, l}$ is the Kronecker delta.
This crystal is denoted by $\z^\infty _{\bm j}$. 

\begin{prop}[{see \cite[Sect.~2.4]{NZ}}]\label{p:star_string}
Let ${\bm i}$ and ${\bm j}$ be as above.
\begin{enumerate}
\item[{\rm (1)}] There exists a unique strict embedding of crystals $\Psi_{\bm j} \colon \mathcal{B}(\infty) \hookrightarrow \z^\infty _{\bm j}$ such that $\Psi_{\bm j} (b_\infty) = (\ldots, 0, \ldots, 0, 0)$, which is called the \emph{Kashiwara embedding} with respect to ${\bm j}$. 
\item[{\rm (2)}] If $(\ldots, a_k, \ldots, a_2, a_1) \in \Psi_{\bm j} (\mathcal{B} (\infty))$, then $a_k = 0$ for all $k > N$. 
\item[{\rm (3)}] For all $w \in W$, it holds that
\[\Psi_{\bm j} (\mathcal{B}_w (\infty)) = \{(\ldots, 0, 0, a_N, \ldots, a_2, a_1) \mid (a_1, \ldots, a_N) \in \Phi_{\bm i} (\mathcal{B}_{w^{-1}} (\infty))\}.\]
\end{enumerate}
\end{prop}

\section{Semi-toric degenerations of Richardson varieties}\label{s:semi-toric}

In this section, we review results by Morier-Genoud \cite{Mor} on semi-toric degenerations of Richardson varieties. 
We first recall relations between string polytopes and Lusztig polytopes, following \cite{Mor}.
By \cite[Proposition 8.2]{Lus3}, each reduced word ${\bm i} \in R(w_0)$ induces a bijective map $b_{\bm i}$ from $\z^N_{\geq 0}$ to the canonical basis $\{G^{\rm low}(b) \mid b \in \mathcal{B}(\infty)\}$, which is called a \emph{Lusztig parametrization}.
For ${\bm i} \in R(w_0)$, let $\Upsilon_{\bm i} \colon \mathcal{B}(\infty) \rightarrow \z^N_{\geq 0}$ denote the bijective map induced by $b_{\bm i} ^{-1}$.
Through the bijective map $\pi_\lambda \colon \widetilde{\mathcal{B}} (\lambda) \xrightarrow{\sim} \mathcal{B} (\lambda)$ in \cref{p:relations_crystals}, the map $\Upsilon_{\bm i}$ induces a \emph{Lusztig parametrization} $\mathcal{B} (\lambda) \hookrightarrow \z^N_{\geq 0}$ of $\mathcal{B}(\lambda)$, which we denote by the same symbol $\Upsilon_{\bm i}$. 
Replacing $\Phi_{\bm i}$ by $\Upsilon_{\bm i}$ in \cref{d:string_polytopes}, we define $\widehat{\mathcal{S}}_{\bm i} (\lambda), \widehat{\mathcal{C}}_{\bm i} (\lambda)$, and $\widehat{\Delta}_{\bm i} (\lambda)$. 
This $\widehat{\Delta}_{\bm i} (\lambda)$ is called the \emph{Lusztig polytope} associated with ${\bm i}$ and $\lambda$ (see \cite{BZ} and \cite[Sect.\ 12]{FFL}).
Let $i \mapsto i^\ast$ denote the involution on $I$ given by $w_0 (\alpha_i) = -\alpha_{i^\ast}$ for $i \in I$.
For ${\bm i} = (i_1, \ldots, i_N) \in R(w_0)$, we set ${\bm i}^\ast \coloneqq (i_1 ^\ast, \ldots, i_N ^\ast) \in R(w_0)$, and define $\Omega_{{\bm i}, \lambda} \colon \r^N \rightarrow \r^N$, $(t_1, \ldots, t_N) \mapsto (t_1^\prime, \ldots, t_N^\prime)$, for $\lambda \in P_\r$ by 
\[t_k^\prime \coloneqq \langle \lambda, h_{i_k} \rangle -t_k -\sum_{k < j \leq N} c_{i_k, i_j} t_j\]
for $1 \leq k \leq N$. 
The map $\Omega_{{\bm i}, \lambda}$ is a unimodular affine transformation. 

\begin{thm}[{see \cite[Corollary 2.17]{Mor}}]\label{t:relation_string_Lusztig}
For ${\bm i} = (i_1, \ldots, i_N) \in R(w_0)$ and $\lambda \in P_+$, the map $\Upsilon_{{\bm i}^\ast} \circ \eta_\lambda \circ \Phi_{\bm i}^{-1}$ coincides with $\Omega_{{\bm i}, \lambda}$ on $\Phi_{\bm i} (\mathcal{B}(\lambda))$. In particular, the following equality holds:
\[\Omega_{{\bm i}, \lambda} (\Delta_{\bm i}(\lambda)) = \widehat{\Delta}_{{\bm i}^\ast} (\lambda).\]
\end{thm}

Morier-Genoud \cite{Mor} used this relation to give semi-toric degenerations of Richardson varieties from Caldero's toric degenerations \cite{Cal} of flag varieties.
For ${\bm i} \in R(w_0)$, we define a subset $\mathscr{S}_{\bm i} \subseteq P_+ \times \z^N$ by 
\[\mathscr{S}_{\bm i} \coloneqq \bigcup_{\lambda \in P_+} \{(\lambda, \Phi_{\bm i} (b)) \mid b \in \mathcal{B}(\lambda)\},\] 
and denote by $\mathscr{C}_{\bm i} \subseteq P_\r \times \r^N$ the smallest real closed cone containing $\mathscr{S}_{\bm i}$. 
For each $\lambda \in P_+$, the following equalities hold:
\begin{align*}
\mathscr{S}_{\bm i} \cap \pi_1 ^{-1} (\z_{>0} \lambda) = \omega_\lambda(\mathcal{S}_{\bm i} (\lambda))\quad {\rm and}\quad \mathscr{C}_{\bm i} \cap \pi_1 ^{-1} (\r_{\geq 0} \lambda) = \omega_\lambda(\mathcal{C}_{\bm i} (\lambda)),
\end{align*}
where $\pi_1 \colon P_\r \times \r^N \rightarrow P_\r$ denotes the first projection, and $\omega_\lambda \colon \r_{\geq 0} \times \r^N \rightarrow P_\r \times \r^N$ is defined by $\omega_\lambda (k, {\bm a}) \coloneqq (k\lambda, {\bm a})$.
In particular, it holds that $\pi_2(\mathscr{C}_{\bm i} \cap \pi_1 ^{-1} (\lambda)) = \Delta_{\bm i} (\lambda)$ for the second projection $\pi_2 \colon P_\r \times \r^N \rightarrow \r^N$.
Note that $\pi_2 (\mathscr{C}_{\bm i})$ coincides with the string cone $\mathcal{C}_{\bm i}$, and we have $\pi_2 (\mathscr{S}_{\bm i}) = \Phi_{\bm i} (\mathcal{B}(\infty))$.

\begin{prop}[{see \cite[Sect.\ 3.2]{BZ} and \cite[Sect.\ 1]{Lit}}] 
The real closed cone $\mathscr{C}_{\bm i}$ is a rational convex polyhedral cone, and the equality $\mathscr{C}_{\bm i} \cap (P_+ \times \z^N) = \mathscr{S}_{\bm i}$ holds.
\end{prop}

Denote by $P_{++} \subseteq P_+$ the set of regular dominant integral weights. 
For $\lambda \in P_{++}$, we see by \eqref{eq:Weyl_volume_formula} that ${\rm Vol}_N(\Delta_{\bm i} (\lambda)) > 0$, that is, $\Delta_{\bm i} (\lambda)$ is $N$-dimensional.
For $v, w \in W$ such that $v \leq w$, we define a subset $\mathscr{S}_{\bm i}(X_w^v) \subseteq P_+ \times \z^N$ by 
\[\mathscr{S}_{\bm i}(X_w^v) \coloneqq \bigcup_{\lambda \in P_+} \{(\lambda, \Phi_{\bm i} (b)) \mid b \in \mathcal{B}_w(\lambda) \cap \mathcal{B}^v(\lambda)\},\] 
and denote by $\mathscr{C}_{\bm i}(X_w^v) \subseteq P_\r \times \r^N$ the smallest real closed cone containing $\mathscr{S}_{\bm i}(X_w^v)$. 
For each $\lambda \in P_+$, we set
\begin{align*}
&\Delta_{\bm i}(\lambda, X_w^v) \coloneqq \pi_2(\mathscr{C}_{\bm i}(X_w^v) \cap \pi_1 ^{-1} (\lambda)).
\end{align*}
It is shown in \cite{Mor} that the cone $\mathscr{C}_{\bm i}(X_w^v)$ induces semi-toric degenerations of $X_w^v$ as follows. 

\begin{thm}[{see the proof of \cite[Proposition 3.5 and Theorem 3.7]{Mor}}]\label{t:semi-toric_degenerations}
Let ${\bm i} \in R(w_0)$, and $v, w \in W$ such that $v \leq w$.
\begin{enumerate}
\item[{\rm (1)}] The set $\mathscr{C}_{\bm i}(X_w^v)$ is a union of faces of $\mathscr{C}_{\bm i}$, and it holds that
\begin{align*}
\mathscr{C}_{\bm i}(X_w^v) \cap (P_+ \times \z^N) &= \mathscr{S}_{\bm i}(X_w^v).
\end{align*}
\item[{\rm (2)}] For $\lambda \in P_+$, the set $\Delta_{\bm i}(\lambda, X_w^v)$ is a union of faces of $\Delta_{\bm i}(\lambda)$.
\item[{\rm (3)}] For $\lambda \in P_{++}$, the Richardson variety $X_w^v$ degenerates into the union of irreducible normal toric varieties corresponding to the faces of $\Delta_{\bm i}(\lambda, X_w^v)$.
\end{enumerate}
\end{thm}

For $w \in W$, we write $\mathcal{S}_{\bm i}(X_w) \coloneqq \pi_2 (\mathscr{S}_{\bm i}(X_w))$ and $\mathcal{C}_{\bm i}(X_w) \coloneqq \pi_2 (\mathscr{C}_{\bm i}(X_w))$.
By the definition of $\mathscr{S}_{\bm i}(X_w)$, it follows that 
\begin{align*}
\mathcal{S}_{\bm i}(X_w) = \bigcup_{\lambda \in P_+} \Phi_{\bm i} (\mathcal{B}_w (\lambda)) = \Phi_{\bm i} (\mathcal{B}_w (\infty)).
\end{align*}

\begin{cor}\label{c:union_of_faces_string_cone}
For ${\bm i} \in R(w_0)$ and $w \in W$, the set $\mathcal{C}_{\bm i}(X_w)$ is a union of faces of $\mathcal{C}_{\bm i}$, and it holds that $\mathcal{C}_{\bm i}(X_w) \cap \z^N = \mathcal{S}_{\bm i}(X_w)$.
\end{cor}

\begin{proof}
Let us apply \cite[Lemma 3.6]{Mor} to $\mathcal{S}_{\bm i}(X_w)$ using \cref{t:semi-toric_degenerations} (1). 
We first take $b \in \mathcal{B}(\infty)$ and $b^\prime \in \mathcal{B}(\infty) \setminus \mathcal{B}_w(\infty)$.
Then there exist $\lambda, \mu \in P_+$ such that $b \in \widetilde{\mathcal{B}}(\lambda)$ and $b^\prime \in \widetilde{\mathcal{B}}(\mu)$ by \eqref{eq:Demazure_infty_limit}.
Since we have $(\lambda, \Phi_{\bm i}(b)) \in \mathscr{S}_{\bm i}$ and $(\mu, \Phi_{\bm i}(b^\prime)) \in \mathscr{S}_{\bm i} \setminus \mathscr{S}_{\bm i}(X_w)$, \cref{t:semi-toric_degenerations} (1) implies that $(\lambda + \mu, \Phi_{\bm i}(b) + \Phi_{\bm i}(b^\prime)) \in \mathscr{S}_{\bm i} \setminus \mathscr{S}_{\bm i}(X_w)$.
Hence there exists $b'' \in \widetilde{\mathcal{B}}(\lambda + \mu) \setminus \widetilde{\mathcal{B}}_w(\lambda + \mu)$ such that $\Phi_{\bm i} (b'') = \Phi_{\bm i}(b) + \Phi_{\bm i}(b^\prime)$.
Since $\widetilde{\mathcal{B}}_w(\lambda + \mu) = \widetilde{\mathcal{B}}(\lambda + \mu) \cap \mathcal{B}_w(\infty)$, we have $b'' \notin \mathcal{B}_w(\infty)$, which implies that $\Phi_{\bm i}(b) + \Phi_{\bm i}(b^\prime) = \Phi_{\bm i} (b'') \notin \Phi_{\bm i} (\mathcal{B}_w (\infty)) = \mathcal{S}_{\bm i}(X_w)$.

We next take $b \in \mathcal{B}_w(\infty)$ and $k \in \z_{> 0}$. 
Then there exists $\lambda \in P_+$ such that $b \in \widetilde{\mathcal{B}}_w(\lambda)$ by \eqref{eq:Demazure_infty_limit}.
Since $(\lambda, \Phi_{\bm i} (b)) \in \mathscr{S}_{\bm i}(X_w)$, we see by \cref{t:semi-toric_degenerations} (1) that $(k\lambda, k\Phi_{\bm i} (b)) \in \mathscr{S}_{\bm i}(X_w)$.
Hence there exists $b^\prime \in \widetilde{\mathcal{B}}_w(k\lambda)$ such that $\Phi_{\bm i}(b^\prime) = k\Phi_{\bm i} (b)$, which implies that $k\Phi_{\bm i} (b) \in \Phi_{\bm i} (\mathcal{B}_w (\infty)) = \mathcal{S}_{\bm i}(X_w)$. 
From these, we can apply \cite[Lemma 3.6]{Mor} to $\mathcal{S}_{\bm i}(X_w)$, which implies that $\mathcal{S}_{\bm i}(X_w)$ is the set of lattice points in a union of faces of $\mathcal{C}_{\bm i}$.
Since $\mathcal{C}_{\bm i}(X_w)$ is the smallest real closed cone containing $\mathcal{S}_{\bm i}(X_w)$, we deduce the assertion of the corollary.
\end{proof}

\section{String parametrizations of opposite Demazure crystals}\label{s:string_opposite}

In this section, we give an explicit description of string parametrizations of opposite Demazure crystals.
To do that, we use the relation (\cref{t:relation_string_Lusztig}) between string polytopes and Lusztig polytopes.
Let ${\bm i} = (i_1, \ldots, i_N) \in R(w_0)$. 
Replacing $\Phi_{\bm i}$ by $\Upsilon_{\bm i}$ in the definitions of $\mathscr{S}_{\bm i}, \mathscr{C}_{\bm i}, \mathscr{S}_{\bm i}(X_w^v), \mathscr{C}_{\bm i}(X_w^v)$, and $\Delta_{\bm i}(\lambda, X_w^v)$, we obtain $\widehat{\mathscr{S}}_{\bm i}, \widehat{\mathscr{C}}_{\bm i}, \widehat{\mathscr{S}}_{\bm i} (X_w^v), \widehat{\mathscr{C}}_{\bm i} (X_w^v)$, and $\widehat{\Delta}_{\bm i}(\lambda, X_w^v)$, respectively.
Then it follows that $\pi_2 (\widehat{\mathscr{S}}_{\bm i}) = \Upsilon_{\bm i} (\mathcal{B}(\infty)) = \z_{\geq 0}^N$ and $\pi_2 (\widehat{\mathscr{C}}_{\bm i}) = \r_{\geq 0}^N$.
Define a unimodular transformation $\Omega_{\bm i} \colon P_\r \times \r^N \rightarrow P_\r \times \r^N$ by $\Omega_{\bm i}(\lambda, {\bm a}) \coloneqq (\lambda, \Omega_{{\bm i}, \lambda}({\bm a}))$ for $(\lambda, {\bm a}) \in P_\r \times \r^N$.
By \cref{t:relation_string_Lusztig}, it follows that $\Omega_{\bm i}(\mathscr{S}_{\bm i}) = \widehat{\mathscr{S}}_{{\bm i}^\ast}$ and $\Omega_{\bm i}(\mathscr{C}_{\bm i}) = \widehat{\mathscr{C}}_{{\bm i}^\ast}$.
In addition, we see by \eqref{eq:involution_longest element_for_Demazure_crystal} that $\Omega_{\bm i}(\mathscr{S}_{\bm i}(X_w^v)) = \widehat{\mathscr{S}}_{{\bm i}^\ast}(X_{w_0 v}^{w_0 w})$ and $\Omega_{\bm i}(\mathscr{C}_{\bm i}(X_w^v)) = \widehat{\mathscr{C}}_{{\bm i}^\ast}(X_{w_0 v}^{w_0 w})$.
Hence the following holds by \cref{t:semi-toric_degenerations}.

\begin{cor}\label{c:semi-toric_degenerations}
Let ${\bm i} \in R(w_0)$, and $v, w \in W$ such that $v \leq w$.
\begin{enumerate}
\item[{\rm (1)}] The set $\widehat{\mathscr{C}}_{\bm i}(X_w^v)$ is a union of faces of $\widehat{\mathscr{C}}_{\bm i}$, and it holds that
\begin{align*}
\widehat{\mathscr{C}}_{\bm i}(X_w^v) \cap (P_+ \times \z^N) &= \widehat{\mathscr{S}}_{\bm i}(X_w^v).
\end{align*}
\item[{\rm (2)}] The set $\widehat{\Delta}_{\bm i}(\lambda, X_w^v)$ is a union of faces of $\widehat{\Delta}_{\bm i}(\lambda)$.
\end{enumerate}
\end{cor}

For $w \in W$, we write $\widehat{\mathcal{S}}_{\bm i}(X_w) \coloneqq \pi_2 (\widehat{\mathscr{S}}_{\bm i}(X_w))$ and $\widehat{\mathcal{C}}_{\bm i}(X_w) \coloneqq \pi_2 (\widehat{\mathscr{C}}_{\bm i}(X_w))$.
By the definition of $\widehat{\mathscr{S}}_{\bm i}(X_w)$, it follows that 
\begin{align*}
\widehat{\mathcal{S}}_{\bm i}(X_w) = \bigcup_{\lambda \in P_+} \Upsilon_{\bm i} (\mathcal{B}_w (\lambda)) = \Upsilon_{\bm i} (\mathcal{B}_w (\infty)).
\end{align*}
In a way similar to the proof of \cref{c:union_of_faces_string_cone}, we obtain the following by \cref{c:semi-toric_degenerations} (1). 

\begin{cor}\label{c:union_of_faces_PBW_cone}
For ${\bm i} \in R(w_0)$ and $w \in W$, the set $\widehat{\mathcal{C}}_{\bm i}(X_w)$ is a union of faces of $\r^N_{\geq 0}$, and it holds that $\widehat{\mathcal{C}}_{\bm i}(X_w) \cap \z^N = \widehat{\mathcal{S}}_{\bm i}(X_w)$.
\end{cor}

For $\lambda \in P_+$, we see by \cite[Sect.~1]{Lit} that the string polytope $\Delta_{\bm i} (\lambda)$ is identical to the set of $(a_1, \ldots, a_N) \in \mathcal{C}_{\bm i}$ satisfying $a_j \le \langle\lambda -a_{j+1} \alpha_{i_{j+1}} - \cdots - a_N \alpha_{i_N}, h_{i_j}\rangle$ for all $1 \leq j \leq N$. 
For $1 \leq j \leq N$, let $F_j(\mathscr{C}_{\bm i})$ denote the facet of $\mathscr{C}_{\bm i}$ given by the equality
\[a_j = \langle \lambda, h_{i_j}\rangle - \sum_{j < k \leq N} a_k \langle\alpha_{i_k}, h_{i_j}\rangle\]
for $(\lambda, (a_1, \ldots, a_N)) \in \mathscr{C}_{\bm i}$.
For ${\bm k} = (k_1, \ldots, k_\ell)$ with $\ell \geq 0$ and with $1 \leq k_1 < \cdots < k_\ell \leq N$, we set 
\[F_{\bm k} (\mathscr{C}_{\bm i}) \coloneqq F_{k_1} (\mathscr{C}_{\bm i}) \cap \cdots \cap F_{k_\ell} (\mathscr{C}_{\bm i});\]
when $\ell = 0$, we think of ${\bm k}$ (resp., $F_{\bm k} (\mathscr{C}_{\bm i})$) as the empty sequence (resp., $\mathscr{C}_{\bm i}$).
Faces $F_j (\Delta_{\bm i}(\lambda))$ and $F_{\bm k} (\Delta_{\bm i}(\lambda))$ of $\Delta_{\bm i}(\lambda)$ are similarly defined. 
For $1 \leq j \leq N$, let $F_j(\widehat{\mathscr{C}}_{\bm i})$ denote the facet of $\widehat{\mathscr{C}}_{\bm i}$ given by $a_j = 0$ for $(\lambda, (a_1, \ldots, a_N)) \in \widehat{\mathscr{C}}_{\bm i}$. 
Replacing $F_j(\mathscr{C}_{\bm i})$ by $F_j(\widehat{\mathscr{C}}_{\bm i})$ in the definition of $F_{\bm k} (\mathscr{C}_{\bm i})$, we obtain $F_{\bm k} (\widehat{\mathscr{C}}_{\bm i})$.
Sets $F_j (\widehat{\Delta}_{\bm i}(\lambda)), F_{\bm k} (\widehat{\Delta}_{\bm i}(\lambda)), F_j (\r^N_{\geq 0}), F_{\bm k} (\r^N_{\geq 0})$ are similarly defined. 
Under the transformation $\Omega_{\bm i}$, the face $F_{\bm k} (\mathscr{C}_{\bm i})$ of $\mathscr{C}_{\bm i}$ corresponds to the face $F_{\bm k} (\widehat{\mathscr{C}}_{{\bm i}^\ast})$ of $\widehat{\mathscr{C}}_{{\bm i}^\ast}$.

\begin{lem}\label{l:opposite_Demazure_union_of_faces}
For $w \in W$, the set $\mathscr{C}_{\bm i}(X^w)$ coincides with a union of faces of $\mathscr{C}_{\bm i}$ of the form $F_{\bm k} (\mathscr{C}_{\bm i})$ for some ${\bm k}$. 
\end{lem}

\begin{proof}
Since we have $\Omega_{\bm i}(\mathscr{C}_{\bm i}(X^w)) = \widehat{\mathscr{C}}_{{\bm i}^\ast}(X_{w_0 w})$, it suffices to prove that the set $\widehat{\mathscr{C}}_{\bm i}(X_w)$ coincides with a union of faces of the form $F_{\bm k} (\widehat{\mathscr{C}}_{\bm i})$ for some ${\bm k}$. 
By \cref{c:union_of_faces_PBW_cone}, the set $\widehat{\mathcal{C}}_{\bm i}(X_w)$ is a union of faces of the form $F_{\bm k} (\r^N_{\geq 0})$ for some ${\bm k}$. 
Writing 
\begin{equation}\label{eq:assumption_Demazure_of_PBW}
\begin{aligned}
\widehat{\mathcal{C}}_{\bm i}(X_w) = F_{{\bm k}^{(1)}} (\r^N_{\geq 0}) \cup \cdots \cup F_{{\bm k}^{(t)}} (\r^N_{\geq 0}),
\end{aligned}
\end{equation}
let us prove that $\widehat{\mathscr{C}}_{\bm i}(X_w) = F_{{\bm k}^{(1)}} (\widehat{\mathscr{C}}_{\bm i}) \cup \cdots \cup F_{{\bm k}^{(t)}} (\widehat{\mathscr{C}}_{\bm i})$.
By the definition of $\widehat{\mathscr{C}}_{\bm i}(X_w)$, it suffices to show that 
\begin{equation}\label{eq:goal_Demazure_of_PBW}
\begin{aligned}
\widehat{\mathscr{S}}_{\bm i}(X_w) = (F_{{\bm k}^{(1)}} (\widehat{\mathscr{C}}_{\bm i}) \cup \cdots \cup F_{{\bm k}^{(t)}} (\widehat{\mathscr{C}}_{\bm i})) \cap (P_+ \times \z^N).
\end{aligned}
\end{equation}
Since $\Upsilon_{\bm i} \colon \mathcal{B}(\infty) \rightarrow \z^N_{\geq 0}$ is bijective, the equality $\widetilde{\mathcal{B}}_w(\lambda) = \widetilde{\mathcal{B}}(\lambda) \cap \mathcal{B}_w(\infty)$ implies that 
\begin{align*}
\widehat{\mathscr{S}}_{\bm i}(X_w) &= \{(\lambda, {\bm a}) \in \widehat{\mathscr{S}}_{\bm i} \mid {\bm a} \in \Upsilon_{\bm i} (\mathcal{B}_w(\infty))\}\\
&= \{(\lambda, {\bm a}) \in \widehat{\mathscr{S}}_{\bm i} \mid {\bm a} \in F_{{\bm k}^{(1)}} (\r^N_{\geq 0}) \cup \cdots \cup F_{{\bm k}^{(t)}} (\r^N_{\geq 0})\}\\
&\quad({\rm by}\ \eqref{eq:assumption_Demazure_of_PBW}\ {\rm since}\ \Upsilon_{\bm i} (\mathcal{B}_w(\infty)) = \widehat{\mathcal{C}}_{\bm i}(X_w) \cap \z^N)\\
&= (F_{{\bm k}^{(1)}} (\widehat{\mathscr{C}}_{\bm i}) \cup \cdots \cup F_{{\bm k}^{(t)}} (\widehat{\mathscr{C}}_{\bm i})) \cap (P_+ \times \z^N), 
\end{align*}
which proves \eqref{eq:goal_Demazure_of_PBW}.
\end{proof}

For ${\bm i} = (i_1, \ldots, i_N) \in R(w_0)$ and $w \in W$, we set
\begin{equation}\label{eq:reduced_word_compatible_subset}
\begin{aligned}
R({\bm i}, w) \coloneqq \{{\bm k} = (k_1, \ldots, k_{\ell(w)}) \mid 1 \leq k_1 < \cdots < k_{\ell(w)} \leq N,\ (i_{k_1}, \ldots, i_{k_{\ell(w)}}) \in R(w)\}.
\end{aligned}
\end{equation}

\begin{ex}
Let $G = SL_4(\c)$, ${\bm i} \coloneqq (2, 1, 2, 3, 2, 1) \in R(w_0)$, and $w \coloneqq s_1 s_2 s_1 = s_2 s_1 s_2 \in W$.
Then we have $R({\bm i}, w) = \{(1, 2, 3), (1, 2, 5), (2, 3, 6), (2, 5, 6)\}$.
\end{ex}

The following is the main result of this section. 

\begin{thm}\label{t:string_parametrization_opposite_Demazure}
For ${\bm i} = (i_1, \ldots, i_N) \in R(w_0)$ and $w \in W$, the following equalities hold:
\begin{align*}
\mathscr{C}_{\bm i}(X^w) = \bigcup_{{\bm k} \in R({\bm i}, w)} F_{\bm k} (\mathscr{C}_{\bm i})\quad \text{and}\quad \widehat{\mathscr{C}}_{\bm i}(X_w) = \bigcup_{{\bm k} \in R({\bm i}^\ast, w_0 w)} F_{\bm k} (\widehat{\mathscr{C}}_{\bm i}).
\end{align*}
\end{thm}

\begin{proof}
Since we have $\Omega_{\bm i}(\mathscr{C}_{\bm i}(X^w)) = \widehat{\mathscr{C}}_{{\bm i}^\ast}(X_{w_0 w})$ and $\Omega_{\bm i}(F_{\bm k}(\mathscr{C}_{\bm i})) = F_{\bm k}(\widehat{\mathscr{C}}_{{\bm i}^\ast})$, it suffices to prove the first equality.
We first take $\lambda \in P_+$ and $b \in \mathcal{B}(\lambda)$ such that $(\lambda, \Phi_{\bm i}(b)) \in F_{\bm k} (\mathscr{C}_{\bm i})$ for some ${\bm k} = (k_1, \ldots, k_{\ell(w)}) \in R({\bm i}, w)$.
Write $\Phi_{\bm i}(b) = (a_1, \ldots, a_N)$, and define $b_0, \ldots, b_N \in \mathcal{B}(\lambda)$ by
\begin{equation}\label{eq:inductive_crystal_elements}
\begin{aligned}
&b_N \coloneqq b_\lambda & & \text{and} & & b_j \coloneqq \tilde{f}_{i_{j+1}}^{a_{j+1}} b_{j+1}\quad \text{for}\quad 0 \leq j \leq N-1. 
\end{aligned}
\end{equation}
By the definition of $\Phi_{\bm i}$, we have $b_0 = b$ and $\varepsilon_{i_j} (b_j) = 0$ for $1 \leq j \leq N$.
For $1 \leq l \leq \ell(w)$, it follows by the condition $(\lambda, \Phi_{\bm i}(b)) \in F_{k_l} (\mathscr{C}_{\bm i})$ that 
\begin{align*}
a_{k_l} &= \langle\lambda, h_{i_{k_l}}\rangle - \sum_{k_l < j \leq N} a_j \langle\alpha_{i_j}, h_{i_{k_l}}\rangle\\
&= \langle{\rm wt}(b_{k_l}), h_{i_{k_l}}\rangle\quad(\text{by\ the\ definition\ of}\ b_{k_l})\\
&= \varphi_{i_{k_l}} (b_{k_l})\quad(\text{by\ the\ definition\ of\ crystals\ since}\ \varepsilon_{i_{k_l}} (b_{k_l}) = 0).
\end{align*}
Hence we see by \cref{p:properties of Demazure} that 
\[
b_{k_{\ell(w)} -1} = \tilde{f}_{i_{k_{\ell(w)}}}^{a_{k_{\ell(w)}}} b_{k_{\ell(w)}} = \tilde{f}_{i_{k_{\ell(w)}}}^{\varphi_{i_{k_{\ell(w)}}} (b_{k_{\ell(w)}})} b_{k_{\ell(w)}} \in \mathcal{B}^{s_{i_{k_{\ell(w)}}}}(\lambda).
\]
Then it holds by \cref{p:properties of Demazure} (1) that $b_{k_{\ell(w)-1}} = \tilde{f}_{i_{k_{\ell(w)-1}+1}}^{a_{k_{\ell(w)-1}+1}} \cdots \tilde{f}_{i_{k_{\ell(w)}-1}}^{a_{k_{\ell(w)}-1}} b_{k_{\ell(w)}-1} \in \mathcal{B}^{s_{i_{k_{\ell(w)}}}}(\lambda)$, which implies by \cref{p:properties of Demazure} again that 
\[
b_{k_{\ell(w)-1} -1} = \tilde{f}_{i_{k_{\ell(w)-1}}}^{a_{k_{\ell(w)-1}}} b_{k_{\ell(w)-1}} = \tilde{f}_{i_{k_{\ell(w)-1}}}^{\varphi_{i_{k_{\ell(w)-1}}} (b_{k_{\ell(w)-1}})} b_{k_{\ell(w)-1}} \in \mathcal{B}^{s_{i_{k_{\ell(w)-1}}} s_{i_{k_{\ell(w)}}}}(\lambda).
\]
Continuing this argument, we obtain that $b \in \mathcal{B}^{s_{i_{k_1}} \cdots s_{i_{k_{\ell(w)}}}}(\lambda) = \mathcal{B}^w(\lambda)$. 
Hence we have $(\lambda, \Phi_{\bm i}(b)) \in \mathscr{S}_{\bm i}(X^w)$, which shows that $F_{\bm k} (\mathscr{C}_{\bm i}) \subseteq \mathscr{C}_{\bm i}(X^w)$ for all ${\bm k} \in R({\bm i}, w)$.

By \cref{l:opposite_Demazure_union_of_faces}, the set $\mathscr{C}_{\bm i}(X^w)$ coincides with a union of faces of the form $F_{\bm k} (\mathscr{C}_{\bm i})$ for some ${\bm k}$. 
Fix ${\bm k} = (k_1, \ldots, k_\ell)$ such that $F_{\bm k} (\mathscr{C}_{\bm i}) \subseteq \mathscr{C}_{\bm i}(X^w)$.
We take $\lambda \in P_{++}$ and $b \in \mathcal{B}^w(\lambda)$ such that $(\lambda, \Phi_{\bm i}(b))$ is included in the relative interior of $F_{\bm k} (\mathscr{C}_{\bm i})$. 
Then we have $(\lambda, \Phi_{\bm i}(b)) \notin F_j (\mathscr{C}_{\bm i})$ for all $j \in \{1, \ldots, N\} \setminus \{k_1, \ldots, k_\ell\}$. 
We write $\Phi_{\bm i}(b) = (a_1, \ldots, a_N)$, and define $b_0, \ldots, b_N \in \mathcal{B}(\lambda)$ by \eqref{eq:inductive_crystal_elements}. 
Since we have $(\lambda, \Phi_{\bm i}(b)) \notin F_1(\mathscr{C}_{\bm i}) \cup \cdots \cup F_{k_1 - 1}(\mathscr{C}_{\bm i})$, it holds that $\varphi_{i_s} (b_{s-1}) > 0$ for all $1 \leq s \leq k_1 - 1$.
Hence the string property of $\mathcal{B}^w(\lambda)$ (\cref{p:string property}) implies that $b_{k_1-1} \in \mathcal{B}^w(\lambda)$.
We set $w_1 \coloneqq w$, and define $w_2, \ldots, w_{\ell+1} \in W$ by 
\[
w_t \coloneqq  
\begin{cases}
s_{i_{k_{t-1}}} w_{t-1} \quad&(w_{t-1} > s_{i_{k_{t-1}}} w_{t-1}),\\
w_{t-1} \quad&(w_{t-1} < s_{i_{k_{t-1}}} w_{t-1})
\end{cases}
\]
for $2 \leq t \leq \ell+1$.
Then it follows by \cref{p:properties of Demazure} that $b_{k_1} \in \mathcal{B}^{w_2}(\lambda)$.
Since it holds that $(\lambda, \Phi_{\bm i}(b)) \notin F_{k_1 + 1} (\mathscr{C}_{\bm i}) \cup \cdots \cup F_{k_2 - 1} (\mathscr{C}_{\bm i})$, we obtain that $\varphi_{i_s} (b_{s-1}) > 0$ for all $k_1 + 1 \leq s \leq k_2 - 1$.
Hence the string property of $\mathcal{B}^{w_2}(\lambda)$ (\cref{p:string property}) implies that $b_{k_2-1} \in \mathcal{B}^{w_2}(\lambda)$.
This implies by \cref{p:properties of Demazure} that $b_{k_2} \in \mathcal{B}^{w_3}(\lambda)$.
Continuing this argument, we conclude that $b_\lambda = b_N \in \mathcal{B}^{w_{\ell+1}}(\lambda)$. 
Since $\lambda \in P_{++}$, this implies that $w_{\ell+1} = e$.
Hence it follows by the definition of $w_{\ell+1}$ that there exists a subsequence $(k_{t_1}, \ldots, k_{t_m})$ of $(k_1, \ldots, k_\ell)$ giving a reduced word $(i_{k_{t_1}}, \ldots, i_{k_{t_m}}) \in R(w)$. 
Since we have
\[F_{\bm k}(\mathscr{C}_{\bm i}) = F_{k_1}(\mathscr{C}_{\bm i}) \cap \cdots \cap F_{k_\ell}(\mathscr{C}_{\bm i}) \subseteq F_{k_{t_1}}(\mathscr{C}_{\bm i}) \cap \cdots \cap F_{k_{t_m}}(\mathscr{C}_{\bm i}),\]
this proves the theorem.
\end{proof}

The following is an immediate consequence of \cref{t:string_parametrization_opposite_Demazure} and the proof of \cref{l:opposite_Demazure_union_of_faces}.

\begin{cor}\label{c:string_parametrization_opposite_Demazure}
For ${\bm i} \in R(w_0)$, $w \in W$, and $\lambda \in P_+$, the equalities $\Delta_{\bm i}(\lambda, X^w) = \bigcup_{{\bm k} \in R({\bm i}, w)} F_{\bm k} (\Delta_{\bm i}(\lambda))$, $\widehat{\Delta}_{\bm i}(\lambda, X_w) = \bigcup_{{\bm k} \in R({\bm i}^\ast, w_0 w)} F_{\bm k} (\widehat{\Delta}_{\bm i}(\lambda))$, and $\widehat{\mathcal{C}}_{\bm i}(X_w) = \bigcup_{{\bm k} \in R({\bm i}^\ast, w_0 w)} F_{\bm k} (\r^N_{\geq 0})$ hold.
\end{cor}

\begin{ex}
Let $G = SL_4(\c)$, ${\bm i} \coloneqq (1, 2, 3, 2, 1, 2) \in R(w_0)$, and $w \coloneqq s_3 s_2 \in W$.
Then we have $R({\bm i}, w) = \{(3, 4), (3, 6)\}$. 
Hence \cref{c:string_parametrization_opposite_Demazure} implies for $\lambda \in P_+$ that 
\begin{align*}
\Delta_{\bm i}(\lambda, X^w) =&\ F_{(3, 4)} (\Delta_{\bm i}(\lambda)) \cup F_{(3, 6)} (\Delta_{\bm i}(\lambda))\\
=&\ \{(a_1, \ldots, a_6) \in \Delta_{\bm i}(\lambda) \mid a_3 = \langle \lambda, h_3\rangle + a_4 + a_6,\ a_4 = \langle \lambda, h_2\rangle + a_5 - 2a_6\}\\ 
&\cup \{(a_1, \ldots, a_6) \in \Delta_{\bm i}(\lambda) \mid a_3 = \langle \lambda, h_3\rangle + a_4 + a_6,\ a_6 = \langle \lambda, h_2\rangle\}.
\end{align*}
\end{ex}

\begin{rem}\label{r:relation_with_PBW_quantum_nilpotent}
For $w \in W$, Kimura \cite[Theorems 4.25 and 4.29]{Kim} proved that the quantum nilpotent subalgebra $U_q^-(w)$ is compatible with the upper global basis and that the corresponding subset $\mathscr{B}(U_q^-(w))$ of $\mathcal{B}(\infty)$ has a PBW-parametrization $\mathscr{B}(U_q^-(w)) = \{b({\bm c}, (i_1, \ldots, i_{\ell(w)}))\}_{{\bm c} \in \z_{\geq 0}^{\ell(w)}}$ for each $(i_1, \ldots, i_{\ell(w)}) \in R(w)$, where we used the notation in \cite[Definition 3.28]{KimOya}.
By the proof of \cite[Theorem 5.13]{Kim}, the set $\mathscr{B}(U_q^-(w))$ is included in $\mathcal{B}_{w^{-1}}(\infty)$. 
We fix $(i_1, \ldots, i_{\ell(w)}) \in R(w)$, and extend it to a reduced word ${\bm i} = (i_1, \ldots, i_{\ell(w)}, i_{\ell(w)+1}, \ldots, i_N) \in R(w_0)$.  
Then we have $s_{i_{\ell(w)+1}} \cdots s_{i_N} = w^{-1} w_0$ and $s_{i_{\ell(w)+1}^\ast} \cdots s_{i_N^\ast} = w_0 w^{-1}$, which implies that $(\ell(w)+1, \ldots, N) \in R({\bm i}^\ast, w_0 w^{-1})$.
Hence the face $F_{(\ell(w)+1, \ldots, N)} (\r^N_{\geq 0})$ of $\r^N_{\geq 0}$ appears in the description of $\widehat{\mathcal{C}}_{\bm i}(X_{w^{-1}})$ in \cref{c:string_parametrization_opposite_Demazure}. 
This face corresponds to the set $\mathscr{B}(U_q^-(w))$ since the definition of $\Upsilon_{\bm i}$ implies that 
\[\Upsilon_{\bm i}(b((c_1, \ldots, c_{\ell(w)}), (i_1, \ldots, i_{\ell(w)}))) = (c_1, \ldots, c_{\ell(w)}, 0, \ldots, 0)\] 
for all $(c_1, \ldots, c_{\ell(w)}) \in \z_{\geq 0}^{\ell(w)}$.
\end{rem}

For a rational convex polytope $F \subseteq \r^N$ with dimension $d$, let $\r F \subseteq \r^N$ denote the affine span of $F$, and ${\rm Vol}_d (F)$ the volume of $F$ with respect to the lattice $\r F \cap \z^N$. 
For $\lambda \in P_{++}$, the line bundle $\mathcal{L}_\lambda$ on $G/B$ is very ample, that is, it induces a closed embedding $G/B \hookrightarrow \mathbb{P}(H^0(G/B, \mathcal{L}_\lambda)^\ast) = \mathbb{P}(V(\lambda))$ (see, for instance, \cite[Sect.\ 1.4]{Bri}). 
For a closed subvariety $Z \subseteq G/B$, we obtain a closed embedding $Z \hookrightarrow \mathbb{P}(V(\lambda))$ by composing $G/B \hookrightarrow \mathbb{P}(V(\lambda))$ with the inclusion map $Z \hookrightarrow G/B$.
Then the volume of $\mathcal{L}_\lambda$ on $Z$ is defined by
\[{\rm Vol}(Z, \mathcal{L}_\lambda) \coloneqq \frac{1}{d!} {\rm deg}(Z \hookrightarrow \mathbb{P}(V(\lambda))),\]
where $d \coloneqq \dim_\c Z$.
Since we have $|\mathcal{B}^w(\lambda)| = \dim_\c H^0(X^w, \mathcal{L}_\lambda)$ for $w \in W$ and $\lambda \in P_+$, we obtain the following by \cref{c:string_parametrization_opposite_Demazure}.

\begin{cor}\label{c:string_opposite_Demazure_volume}
For $w \in W$ and $\lambda \in P_+$, the dimension $\dim_\c H^0(X^w, \mathcal{L}_\lambda)$ equals the cardinality of $\bigcup_{{\bm k} \in R({\bm i}, w)} F_{\bm k} (\Delta_{\bm i}(\lambda)) \cap \z^N$.
In particular, it holds for $\lambda \in P_{++}$ that
\[{\rm Vol}(X^w, \mathcal{L}_\lambda) = \sum_{{\bm k} \in R({\bm i}, w)} {\rm Vol}_{N-\ell(w)} (F_{\bm k} (\Delta_{\bm i}(\lambda))).\]
\end{cor}

For $\lambda \in P_{++}$, let $X(\Delta_{\bm i}(\lambda))$ denote the normal toric variety corresponding to $\Delta_{\bm i}(\lambda)$.
The following is an immediate consequence of \cref{t:semi-toric_degenerations} (3) and \cref{c:string_parametrization_opposite_Demazure}.

\begin{cor}
For $w \in W$ and $\lambda \in P_{++}$, the opposite Schubert variety $X^w$ degenerates into the union $\bigcup_{{\bm k} \in R({\bm i}, w)} X(F_{\bm k} (\Delta_{\bm i}(\lambda)))$ of irreducible normal toric subvarieties $X(F_{\bm k} (\Delta_{\bm i}(\lambda)))$ of $X(\Delta_{\bm i}(\lambda))$ corresponding to the faces $F_{\bm k} (\Delta_{\bm i}(\lambda))$, ${\bm k} \in R({\bm i}, w)$.
\end{cor}

\section{Kogan faces and Demazure crystals of type $A$}\label{s:type_A}

In this section, we consider the case $G = SL_{n+1}(\c)$, and relate results in \cite{KST} with crystal bases.
We identify the set $I$ of vertices of the Dynkin diagram with $\{1, \ldots, n\}$ as follows:
\begin{align*}
&A_n\ \begin{xy}
\ar@{-} (50,0) *++!D{1} *\cir<3pt>{};
(60,0) *++!D{2} *\cir<3pt>{}="C"
\ar@{-} "C";(65,0) \ar@{.} (65,0);(70,0)^*!U{}
\ar@{-} (70,0);(75,0) *++!D{n-1} *\cir<3pt>{}="D"
\ar@{-} "D";(85,0) *++!D{n} *\cir<3pt>{}="E"
\end{xy}.
\end{align*}
Then the Weyl group $W$ is isomorphic to the symmetric group $\mathfrak{S}_{n+1}$ by identifying the simple reflection $s_i$ with the transposition $(i\ i+1)$. 
We set $N \coloneqq \ell(w_0) = \frac{n(n +1)}{2}$, and define ${\bm i}_A = (i_1, \ldots, i_N) \in R(w_0)$ by
\begin{equation}\label{eq:reduced_word_type_A}
\begin{aligned}
{\bm i}_A \coloneqq (1, 2, 1, 3, 2, 1, \ldots, n, n-1, \ldots, 1).
\end{aligned}
\end{equation}
Then Littelmann \cite[Theorem 5.1]{Lit} proved that the string cone $\mathcal{C}_{{\bm i}_A}$ coincides with the set of 
\begin{equation}\label{eq:coordinate_type_A_string}
\begin{aligned}
(a_1 ^{(1)}, a_1 ^{(2)}, a_2 ^{(1)}, a_1 ^{(3)}, a_2 ^{(2)}, a_3 ^{(1)}, \ldots, a_1 ^{(n)}, \ldots, a_n ^{(1)}) \in \r^N
\end{aligned}
\end{equation}
satisfying the following inequalities:
\begin{align*}
&a_1 ^{(1)} \geq 0, & &a_1 ^{(2)} \geq a_2 ^{(1)} \geq 0, & &\ldots, & &a_1 ^{(n)} \geq \cdots \geq a_n ^{(1)} \geq 0.
\end{align*}
We arrange the equations of the facets of $\mathcal{C}_{{\bm i}_A}$ as 
\begin{align*}
a_1 ^{(1)} = 0,\ a_2 ^{(1)} = 0,\ a_1 ^{(2)} = a_2 ^{(1)},\ a_3 ^{(1)} = 0,\ \ldots,\ a_2 ^{(n-1)} = a_3 ^{(n-2)},\ a_1 ^{(n)} = a_2 ^{(n-1)},
\end{align*}
and denote the corresponding facets of $\mathcal{C}_{{\bm i}_A}$ by $F_1^\vee(\mathcal{C}_{{\bm i}_A}), \ldots, F_N^\vee(\mathcal{C}_{{\bm i}_A})$, respectively. 
For $1 \leq j \leq N$ and $\lambda \in P_+$, let $F_j^\vee(\Delta_{{\bm i}_A}(\lambda))$ be the face of $\Delta_{{\bm i}_A}(\lambda)$ given by $F_j^\vee(\Delta_{{\bm i}_A}(\lambda)) \coloneqq \Delta_{{\bm i}_A}(\lambda) \cap F_j^\vee(\mathcal{C}_{{\bm i}_A})$.
For ${\bm k} = (k_1, \ldots, k_\ell)$ with $\ell \geq 0$ and with $1 \leq k_1 < \cdots < k_\ell \leq N$, we write $F_{\bm k}^\vee(\mathcal{C}_{{\bm i}_A}) \coloneqq F_{k_1}^\vee(\mathcal{C}_{{\bm i}_A}) \cap \cdots \cap F_{k_\ell}^\vee(\mathcal{C}_{{\bm i}_A})$ and $F_{\bm k}^\vee(\Delta_{{\bm i}_A}(\lambda)) \coloneqq F_{k_1}^\vee(\Delta_{{\bm i}_A}(\lambda)) \cap \cdots \cap F_{k_\ell}^\vee(\Delta_{{\bm i}_A}(\lambda))$.
Faces $F_j^\vee(\mathscr{C}_{{\bm i}_A})$ and $F_{\bm k}^\vee(\mathscr{C}_{{\bm i}_A})$ of $\mathscr{C}_{{\bm i}_A}$ are similarly defined.

For $\lambda \in P_+$ and $1 \leq k \leq n+1$, we set $a_k ^{(0)} \coloneqq \sum_{k \leq \ell \leq n} \langle \lambda, h_\ell \rangle$, where $a_{n+1}^{(0)}$ is regarded as $0$. 
Then the \emph{Gelfand--Tsetlin polytope} $GT(\lambda)$ is defined to be the set of \eqref{eq:coordinate_type_A_string} satisfying the following inequalities:
\[\begin{matrix}
a_1 ^{(0)} & & a_2 ^{(0)} & & \cdots & & & a_n ^{(0)} & & a_{n+1} ^{(0)}\\
 & a_1 ^{(1)} & & a_2 ^{(1)} & & \cdots & & & a_n ^{(1)} & \\
 & & a_1 ^{(2)} & & \cdots & & & a_{n -1} ^{(2)} & & \\
 & & & \ddots & & \ldots & & & & \\
 & & & & a_1 ^{(n-1)} & & a_2 ^{(n-1)} & & & \\
 & & & & & a_1 ^{(n)}, & & & & 
\end{matrix}\]
where we mean by the notation 
\[\begin{matrix}
a & & c\\
 & b & 
\end{matrix}\]
that $a \geq b \geq c$. 
The Gelfand--Tsetlin polytope $GT(\lambda)$ is an integral polytope for all $\lambda \in P_+$, and it is $N$-dimensional if $\lambda \in P_{++}$.
A face of $GT(\lambda)$ given by equations of the type $a_k^{(l)} = a_k^{(l+1)}$ (resp., $a_k^{(l)} = a_{k+1}^{(l-1)}$) is called a \emph{Kogan face} (resp., a \emph{dual Kogan face}).
We arrange the equations of the type $a_k^{(l)} = a_{k+1}^{(l-1)}$ as 
\begin{align*}
a_2 ^{(n-1)} = a_1 ^{(n)},\ a_3 ^{(n-2)} = a_2 ^{(n-1)},\ a_2 ^{(n-2)} = a_1 ^{(n-1)},\ a_4 ^{(n-3)} = a_3 ^{(n-2)},\ \ldots,\ a_3 ^{(0)} = a_2 ^{(1)},\ a_2 ^{(0)} = a_1 ^{(1)},
\end{align*}
and denote the corresponding facets of $GT(\lambda)$ by $F_1 (GT(\lambda)), \ldots, F_N (GT(\lambda))$, respectively.
We also arrange the equations of the type $a_k^{(l)} = a_k^{(l+1)}$ as 
\begin{align*}
a_1 ^{(n-1)} = a_1 ^{(n)},\ a_1 ^{(n-2)} = a_1 ^{(n-1)},\ a_2 ^{(n-2)} = a_2 ^{(n-1)},\ a_1 ^{(n-3)} = a_1 ^{(n-2)},\ \ldots,\ a_{n-1} ^{(0)} = a_{n-1} ^{(1)},\ a_n ^{(0)} = a_n ^{(1)},
\end{align*}
and denote the corresponding facets of $GT(\lambda)$ by $F^\vee_1(GT(\lambda)), \ldots, F^\vee_N(GT(\lambda))$, respectively. 
For ${\bm k} = (k_1, \ldots, k_\ell)$ with $\ell \geq 0$ and with $1 \leq k_1 < \cdots < k_\ell \leq N$, we define $F_{\bm k}(GT(\lambda))$ (resp., $F_{\bm k}^\vee(GT(\lambda))$) in a way similar to $F_{\bm k}(\Delta_{{\bm i}_A}(\lambda))$ (resp., $F_{\bm k}^\vee(\Delta_{{\bm i}_A}(\lambda))$).

\begin{defi}[{see \cite[Sect.\ 2.2.1]{Kog} and \cite[Sects.\ 3.3 and 4.3]{KST}}]
A Kogan face $F_{\bm k}^\vee(GT(\lambda))$ (resp., a dual Kogan face $F_{\bm k}(GT(\lambda))$) of $GT(\lambda)$ is said to be \emph{reduced} if $(i_{k_1}, \ldots, i_{k_\ell})$ is a reduced word. 
In this case, we set $w(F_{\bm k}^\vee(GT(\lambda))) \coloneqq w_0 s_{i_{k_1}} \cdots s_{i_{k_\ell}} w_0$ (resp., $w(F_{\bm k}(GT(\lambda))) \coloneqq w_0 s_{i_{k_1}} \cdots s_{i_{k_\ell}} w_0$).
\end{defi}

Littelmann \cite[Corollary 5]{Lit} gave a unimodular affine transformation from the string polytope $\Delta_{{\bm i}_A}(\lambda)$ to $GT(\lambda)$. 
Under this transformation, the face $F_{\bm k}(\Delta_{{\bm i}_A}(\lambda))$ (resp., $F_{\bm k}^\vee(\Delta_{{\bm i}_A}(\lambda))$) of $\Delta_{{\bm i}_A}(\lambda)$ corresponds to the dual Kogan face $F_{\bm k}(GT(\lambda))$ (resp., the Kogan face $F_{\bm k}^\vee(GT(\lambda))$) of $GT(\lambda)$. 
Hence we obtain the following as an immediate consequence of \cref{c:string_parametrization_opposite_Demazure}.

\begin{cor}\label{c:dual_Kogan_opposite_Demazure}
For $w \in W$ and $\lambda \in P_+$, the union $\bigcup_{{\bm k} \in R({\bm i}_A, w)} F_{\bm k}(GT(\lambda))$ of reduced dual Kogan faces $F_{\bm k}(GT(\lambda))$ with $w(F_{\bm k}(GT(\lambda))) = w_0 w w_0$ gives a polyhedral parametrization of $\mathcal{B}^w(\lambda)$ under the unimodular affine transformation $GT(\lambda) \simeq \Delta_{{\bm i}_A}(\lambda)$.
\end{cor}

\begin{rem}
Such relation between reduced dual Kogan faces and the character of $\mathcal{B}^w(\lambda)$ was previously given in \cite[Corollary 5.2]{KST}.
\end{rem}

Let $\mathcal{PD}_n$ denote the set of subsets of 
\[Y_n \coloneqq \{(i, j) \in \z^2 \mid 1 \leq i \leq n,\ 1 \leq j \leq n-i+1\}.\]
We regard $Y_n$ as a Young diagram, and represent an element $D \in \mathcal{PD}_n$ by putting $+$ in the boxes contained in $D$. 
For instance, a set $D = \{(1, 3), (2, 1), (3, 1), (4, 1)\} \in \mathcal{PD}_4$ is represented as 
\[D = \begin{ytableau}
\mbox{} & \mbox{} & + & \mbox{} \\
+ & \mbox{} & \mbox{} & \none \\
+ & \mbox{} & \none & \none \\
+ & \none & \none & \none
\end{ytableau}.\]
An element of $\mathcal{PD}_n$ is called a \emph{pipe dream} (see \cite[Sect.\ 1.4]{KnM}).
Arrange the elements of $Y_n$ as 
\[((n, 1), (n-1, 1), (n-1, 2), (n-2, 1), \ldots, (1, n-1), (1, n)).\]
Then we associate to each $D \in \mathcal{PD}_n$ a sequence ${\bm k}_D = (k_1, \ldots, k_{|D|})$ with $1 \leq k_1 < \cdots < k_{|D|} \leq N$ by arranging $1 \leq k \leq N$ such that the $k$-th element of $Y_n$ is included in $D$.
For instance, a pipe dream $D = \{(1, 3), (2, 1), (3, 1), (4, 1)\} \in \mathcal{PD}_4$ gives a sequence ${\bm k}_D = (1, 2, 4, 9)$.
For $D \in \mathcal{PD}_n$, the Kogan face $F_{{\bm k}_D}^\vee(GT(\lambda))$ is reduced if and only if the pipe dream $D$ is reduced in the sense of \cite[Definition 1.4.3]{KnM} (see \cite[Lemma 1.4.5]{KnM}).
A reduced pipe dream is the same as an \emph{rc-graph} invented in \cite{FK} and developed in \cite{BB}.
In order to relate reduced Kogan faces with Demazure crystals, we review the notion of \emph{ladder moves} introduced in \cite{BB}.

\begin{defi}[{see \cite[Sect.\ 3]{BB}}]
For $(i, j) \in Y_n$, the \emph{ladder move} $L_{i, j}$ is defined as follows.
Take $D \in \mathcal{PD}_n$ satisfying the following conditions:
\begin{itemize}
\item $(i, j) \in D$, $(i, j+1) \notin D$,
\item there exists $1 \leq m < i$ such that $(i-m, j), (i-m, j+1) \notin D$ and $(i-k, j), (i-k, j+1) \in D$ for all $1 \leq k < m$.
\end{itemize}
Then we define $L_{i, j} (D) \in \mathcal{PD}_n$ by 
\[L_{i, j} (D) \coloneqq D \cup \{(i-m, j+1)\} \setminus \{(i, j)\}.\]
Denote the pipe dream $L_{i, j} (D)$ also by $L_{i-m, j}^\prime(D)$.
\end{defi}

\begin{ex}\label{ex:ladder_moves_type_A}
Let $n = 4$. Then we obtain the following: 
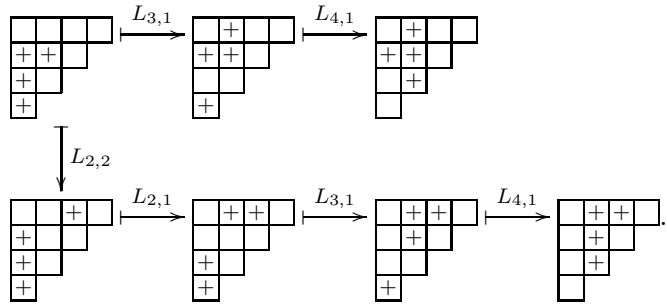
\begin{align*}
\xymatrix{
{\begin{ytableau}
\mbox{} & \mbox{} & \mbox{} & \mbox{} \\
+ & + & \mbox{} & \none \\
+ & \mbox{} & \none & \none \\
+ & \none & \none & \none
\end{ytableau}} \ar@{|->}[d]^-{L_{2, 2}} \ar@{|->}[r]^-{L_{3, 1}} & {\begin{ytableau}
\mbox{} & + & \mbox{} & \mbox{} \\
+ & + & \mbox{} & \none \\
\mbox{} & \mbox{} & \none & \none \\
+ & \none & \none & \none
\end{ytableau}} \ar@{|->}[r]^-{L_{4, 1}} & {\begin{ytableau}
\mbox{} & + & \mbox{} & \mbox{} \\
+ & + & \mbox{} & \none \\
\mbox{} & + & \none & \none \\
\mbox{} & \none & \none & \none
\end{ytableau}} & \\
{\begin{ytableau}
\mbox{} & \mbox{} & + & \mbox{} \\
+ & \mbox{} & \mbox{} & \none \\
+ & \mbox{} & \none & \none \\
+ & \none & \none & \none
\end{ytableau}} \ar@{|->}[r]^-{L_{2, 1}} & {\begin{ytableau}
\mbox{} & + & + & \mbox{} \\
\mbox{} & \mbox{} & \mbox{} & \none \\
+ & \mbox{} & \none & \none \\
+ & \none & \none & \none
\end{ytableau}} \ar@{|->}[r]^-{L_{3, 1}} & {\begin{ytableau}
\mbox{} & + & + & \mbox{} \\
\mbox{} & + & \mbox{} & \none \\
\mbox{} & \mbox{} & \none & \none \\
+ & \none & \none & \none
\end{ytableau}} \ar@{|->}[r]^-{L_{4, 1}} & {\begin{ytableau}
\mbox{} & + & + & \mbox{} \\
\mbox{} & + & \mbox{} & \none \\
\mbox{} & + & \none & \none \\
\mbox{} & \none & \none & \none
\end{ytableau}.} 
}
\end{align*}
\end{ex}

We set
\[(p_1, \ldots, p_N) \coloneqq (n, \underbrace{n-1, n-1}_2, \underbrace{n-2, n-2, n-2}_3, \ldots, \underbrace{1, \ldots, 1}_n).\]
Then the sequence $((p_1, i_1), \ldots, (p_N, i_N))$ gives an arrangement of the elements of $Y_n$, which is different from the arrangement above.
We associate to each $D \in \mathcal{PD}_n$ a sequence ${\bm k}_D^\prime = (k_1, \ldots, k_{N-|D|})$ with $1 \leq k_1 < \cdots < k_{N-|D|} \leq N$ by arranging $1 \leq k \leq N$ such that $(p_k, i_k) \in Y_n \setminus D$.
For instance, a pipe dream $D = \{(1, 3), (2, 1), (3, 1), (4, 1)\} \in \mathcal{PD}_4$ corresponds to a sequence ${\bm k}_D^\prime = (2, 4, 5, 7, 9, 10)$.
For $w \in W$, define $D(w) \in \mathcal{PD}_n$ by the condition that ${\bm k}_{D(w)}^\prime$ is the minimum element in $R({\bm i}_A, w)$ with respect to the lexicographic order. 
If $w = e$, then we think of $D(w)$ as $\mathcal{PD}_n$.

\begin{prop}\label{p:relation_with_BB}
For $w \in W$, the equality 
\[D(w) = \{(i, j) \mid 1 \leq i \leq n,\ 1 \leq j \leq m(i)\}\]
holds, where $m(i)$ is the cardinality of $\{i < j \leq n+1 \mid w^{-1} w_0 (j) < w^{-1} w_0 (i)\}$ for $1 \leq i \leq n$.
\end{prop}

\begin{proof}
We proceed by induction on $\ell(w)$.
If $\ell(w) = 0$, then the assertion is obvious. 
Assume that $\ell \coloneqq \ell(w) > 0$, and write ${\bm k}_{D(w)}^\prime = (k_1, \ldots, k_\ell)$. 
If we set $i \coloneqq i_{k_\ell}$ and $w^\prime \coloneqq w s_i$, then we have ${\bm k}_{D(w^\prime)}^\prime = (k_1, \ldots, k_{\ell-1})$, and it follows that 
\[D(w) = D(w^\prime) \setminus \{(p_{k_\ell}, i)\}.\]
By the induction hypothesis, we have $D(w^\prime) = \{(i, j) \mid 1 \leq i \leq n,\ 1 \leq j \leq m^\prime(i)\}$, where $m^\prime(i)$ is the cardinality of $\{i < j \leq n+1 \mid (w^\prime)^{-1} w_0 (j) < (w^\prime)^{-1} w_0(i)\}$ for $1 \leq i \leq n$.
Since we have $(w^\prime)^{-1} w_0 = s_i w^{-1} w_0$ and $\ell((w^\prime)^{-1} w_0) = \ell(w^{-1} w_0) + 1$, the set $\{(i, j) \mid 1 \leq i \leq n,\ 1 \leq j \leq m(i)\}$ is obtained from $D(w^\prime)$ by removing $(((w^\prime)^{-1} w_0)^{-1} (i+1), m^\prime(((w^\prime)^{-1} w_0)^{-1} (i+1))) \in D(w^\prime)$.
Hence it suffices to show that $((w^\prime)^{-1} w_0)^{-1} (i+1) = p_{k_\ell}$ and $m^\prime (p_{k_\ell}) = i$.

Since we have $(p_{k_\ell}, i) \in D(w^\prime)$, it holds that $i \leq m^\prime(p_{k_\ell})$.
Suppose for a contradiction that $i < m^\prime (p_{k_\ell})$. 
Consider words obtained from $(i_{k_1}, \ldots, i_{k_\ell})$ by moving $i_{k_\ell} = i$ to the left via $2$-moves, i.e., moves of the form $(i, j) \mapsto (j, i)$ with $c_{i, j} = 0$.
Since these are reduced words, we have $m^\prime (p_{k_\ell} + 1) \neq i-1$ by the definition of ${\bm k}_{D(w^\prime)}^\prime$.
If $i-1 < m^\prime (p_{k_\ell} + 1)$, then we can find an element of $R({\bm i}_A, w)$ which is smaller than $(k_1, \ldots, k_\ell)$ with respect to the lexicographic order; this contradicts to the definition of $D(w)$. 
Hence it holds that $m^\prime (p_{k_\ell} + 1) < i - 1$. 
Then the following reduced word is obtained from $(i_{k_1}, \ldots, i_{k_\ell})$ by moving $i_{k_\ell} = i$ to the left via $2$-moves:
\[(i_{k_1}, \ldots, i_{k_{t_1-2}}, i, i - 1, i, i_{k_{t_1 + 1}}, \ldots, i_{k_{\ell - 1}}),\]
where $t_1$ is defined by $(p_{k_{t_1}}, i_{k_{t_1}}) = (p_{k_\ell} + 1, i - 1)$. 
Replacing $(i, i - 1, i)$ by $(i - 1, i, i - 1)$, we obtain a reduced word $(i_{k_1}, \ldots, i_{k_{t_1-2}}, i - 1, i, i - 1, i_{k_{t_1 + 1}}, \ldots, i_{k_{\ell - 1}}) \in R(w)$.
We next consider words obtained from this by moving $i - 1$ in the $(t_1-1)$-th position to the left via $2$-moves. 
Then the argument above implies that $m^\prime (p_{k_\ell} + 2) < i - 2$, and we obtain a reduced word 
\[(i_{k_1}, \ldots, i_{k_{t_2-2}}, i - 1, i - 2, i - 1, i_{k_{t_2 + 1}}, \ldots, i_{k_{t_1-2}}, i, i - 1, i_{k_{t_1 + 1}}, \ldots, i_{k_{\ell - 1}}),\]
where $t_2$ is defined by $(p_{k_{t_2}}, i_{k_{t_2}}) = (p_{k_\ell} + 2, i - 2)$. 
Continuing this argument, we have $m^\prime (p_{k_\ell} + i - 1) = m^\prime (p_{k_\ell} + i) = 0$, and obtain a reduced word of the form 
\begin{equation}\label{eq:reduced_word_2-moves_last_step}
\begin{aligned}
(i_{k_1}, \ldots, i_{k_{t_i}}, 1, i_{k_{t_i} + 1}, \ldots, i_{k_{t_{i - 1}-2}}, 2, 1, i_{k_{t_{i - 1} + 1}}, \ldots, i_{k_{t_1-2}}, i, i - 1, i_{k_{t_1 + 1}}, \ldots, i_{k_{\ell - 1}}),
\end{aligned}
\end{equation}
where $t_i$ is defined by $(p_{k_{t_i}}, i_{k_{t_i}}) = (p_{k_\ell} + i, 1)$. 
However, since $i_{k_{t_i}} = 1$, the word \eqref{eq:reduced_word_2-moves_last_step} is not reduced, which gives a contradiction. 
Thus, we deduce that $m^\prime (p_{k_\ell}) = i$.

By the definition of ${\bm k}_{D(w^\prime)}^\prime$, we have $(p_k, i_k) \in D(w^\prime)$ for all $k_{\ell-1} < k \leq N$.
Since $k_{\ell-1} < k_\ell$, it holds that $m^\prime(t) = n - t + 1$ for $1 \leq t \leq p_{k_\ell} - 1$. 
Hence we have $(w^\prime)^{-1} w_0 (t) = n - t + 2$ for $1 \leq t \leq p_{k_\ell} - 1$ and $(w^\prime)^{-1} w_0 (p_{k_\ell}) = i + 1$, which implies that $((w^\prime)^{-1} w_0)^{-1} (i+1) = p_{k_\ell}$. 
This proves the proposition.
\end{proof}

Let $\mathscr{L}(D(w))$ denote the set of elements of $\mathcal{PD}_n$ obtained from $D(w)$ by applying sequences of ladder moves.

\begin{cor}\label{c:last_box_for_w_prime}
For $w \in W$ with $\ell \coloneqq \ell(w) > 0$, write ${\bm k}_{D(w)}^\prime = (k_1, \ldots, k_\ell)$ and $w^\prime \coloneqq w s_{i_{k_\ell}}$. 
Then it holds that $(p_{k_\ell}, i_{k_\ell}+1) \notin D$ and 
\[\{(i, j) \mid 1 \leq i \leq p_{k_\ell} - 1,\ 1 \leq j \leq n + 1 -i\} \cup \{(p_{k_\ell}, j) \mid 1 \leq j \leq i_{k_\ell}\} \subseteq D\]
for all $D \in \mathscr{L}(D(w^\prime))$.
\end{cor}

\begin{proof}
By \cref{p:relation_with_BB}, we see that $p_{k_\ell} = n-i_{k_\ell}+1$ or $(p_{k_\ell}, i_{k_\ell}+1) = (p_{k_{\ell-1}}, i_{k_{\ell-1}})$. 
This implies the assertion by the definition of ladder moves. 
\end{proof}

By \cref{p:relation_with_BB}, the set $D(w)$ coincides with $D_{\rm bot}(w^{-1} w_0)$ in \cite[Sect.\ 3]{BB}. Hence we obtain the following. 

\begin{thm}[{see \cite[Theorem 3.7 (c)]{BB}}]\label{t:ladder_moves_BB}
For $w \in W$, the set $\mathscr{L}(D(w))$ coincides with the set of reduced pipe dreams $D$ such that $w(F_{{\bm k}_D}^\vee (GT(\lambda))) = w_0 w$.
\end{thm}

\begin{ex}\label{ex:length_one_type_A}
For $1 \leq i \leq n$, we have $D(s_i) = Y_n \setminus \{(n-i+1, i)\}$ and $\mathscr{L}(D(s_i)) = \{D(s_i)\}$. 
\end{ex}

\begin{ex}
Let $n = 2$. Then we see that 
\begin{align*}
&\mathscr{L}(D(s_1)) = \left\{\begin{ytableau}
+ & +\\
\mbox{} & \none 
\end{ytableau}\right\}, & &\mathscr{L}(D(s_2)) = \left\{\begin{ytableau}
+ & \mbox{}\\
+ & \none 
\end{ytableau}\right\},\\
&\mathscr{L}(D(s_1 s_2)) = \left\{\begin{ytableau}
+ & \mbox{}\\
\mbox{} & \none 
\end{ytableau}\right\}, & &\mathscr{L}(D(s_2 s_1)) = \left\{\begin{ytableau}
\mbox{} & \mbox{}\\
+ & \none 
\end{ytableau}, \begin{ytableau}
\mbox{} & +\\
\mbox{} & \none 
\end{ytableau}\right\}.
\end{align*}
\end{ex}

\begin{ex}
Let $n = 4$, and $w = s_2 s_3 s_4 s_3 s_2 s_1$. Then it follows that 
\[D(w) = \begin{ytableau}
\mbox{} & \mbox{} & \mbox{} & \mbox{} \\
+ & + & \mbox{} & \none \\
+ & \mbox{} & \none & \none \\
+ & \none & \none & \none
\end{ytableau}.\]
Hence the set $\mathscr{L}(D(w))$ consists of the seven pipe dreams given in \cref{ex:ladder_moves_type_A}.
\end{ex}

We also review the notion of (the transpose of) mitosis operators introduced in \cite{KnM}.
For $1 \leq j \leq n$ and $D \in \mathcal{PD}_n$, we set 
\begin{align*}
&{\rm start}_j^\top(D) \coloneqq \min\{1 \leq i \leq n-j+1 \mid (i, j) \notin D\} \cup \{n-j+2\}\ {\rm and}\\
&\mathcal{J}_j^\top(D) \coloneqq \{1 \leq i < {\rm start}_j^\top(D) \mid (i, j+1) \notin D\}.
\end{align*}
For each $i \in \mathcal{J}_j^\top(D)$, we define a pipe dream $D_i(j) \in \mathcal{PD}_n$ as follows: write $\{1 \leq p \leq i \mid (p, j+1) \notin D\} = \{p_0 < p_1 < \cdots < p_r = i\}$, and set 
\[D_i(j) \coloneqq L_{p_r, j} \cdots L_{p_1, j}(D \setminus \{(p_0, j)\}).\]

\begin{defi}[{see \cite[Definition 1.6.1]{KnM}}]
For $1 \leq j \leq n$ and $D \in \mathcal{PD}_n$, the \emph{transposed mitosis operator} ${\rm mitosis}_j^\top$ sends $D$ to 
\[{\rm mitosis}_j^\top(D) \coloneqq \{D_i(j) \mid i \in \mathcal{J}_j^\top(D)\}.\]
If $\mathcal{J}_j^\top(D) = \emptyset$, then we regard ${\rm mitosis}_j^\top(D)$ as the empty set. 
For a subset $\mathcal{A} \subseteq \mathcal{PD}_n$, we set ${\rm mitosis}_j^\top(\mathcal{A}) \coloneqq \bigcup_{D \in \mathcal{A}} {\rm mitosis}_j^\top(D)$.
\end{defi}

The following theorem is obtained by transposing the corresponding property of usual mitosis operators.  

\begin{thm}[{see \cite[Theorem C]{KnM} and \cite[Theorem 15]{Mil}}]\label{t:mitosis_formula}
For $w \in W$ and $(j_1, \ldots, j_{\ell(w)}) \in R(w)$, it holds that
\[\mathscr{L}(D(w)) = {\rm mitosis}_{j_{\ell(w)}}^\top \cdots {\rm mitosis}_{j_1}^\top (Y_n).\]
\end{thm}

Let us consider another operator $M_j$ for $1 \leq j \leq n$, which is similar to ${\rm mitosis}_j^\top$. 
Take $D \in \mathcal{PD}_n$, and assume that $\mathcal{J}_j^\top(D) \neq \emptyset$. 
We denote by $p_0$ the minimum element of $\mathcal{J}_j^\top(D)$, and by $M_j (D)$ the set of elements of $\mathcal{PD}_n$ obtained from $D \setminus \{(p_0, j)\}$ by applying sequences of ladder moves $L_{p, q}$ such that $q = j$.
For a subset $\mathcal{A} \subseteq \mathcal{PD}_n$ such that $\mathcal{J}_j^\top(D) \neq \emptyset$ for all $D \in \mathcal{A}$, we set $M_j(\mathcal{A}) \coloneqq \bigcup_{D \in \mathcal{A}} M_j(D)$; it follows by definition that ${\rm mitosis}_j^\top(\mathcal{A}) \subseteq M_j(\mathcal{A})$.
For $w \in W$, we define a set $\mathscr{M}(w)$ of pipe dreams by 
\[\mathscr{M}(w) \coloneqq M_{i_{k_\ell}} \cdots M_{i_{k_1}} (Y_n),\]
where we write ${\bm k}^\prime_{D(w)} = (k_1, \ldots, k_\ell)$.
Note that $M_{i_{k_r}} \cdots M_{i_{k_1}} (Y_n) = \mathscr{M}(s_{i_{k_1}} \cdots s_{i_{k_r}})$ for each $1 \leq r \leq \ell$.
By \cref{c:last_box_for_w_prime}, the set $\mathscr{M}(w)$ is well-defined and coincides with the set of pipe dreams of the form $\widetilde{L}_\ell \widetilde{L}_{\ell-1} \cdots \widetilde{L}_2 D(w)$, where $\widetilde{L}_r$ for $2 \leq r \leq \ell$ is given by a sequence of ladder moves $L_{p, i_{k_r}}$ such that $p_{k_r} < p$.
In particular, it holds that $\mathscr{M}(w) \subseteq \mathscr{L}(D(w))$. 
Since we have ${\rm mitosis}_{i_{k_\ell}}^\top \cdots {\rm mitosis}_{i_{k_1}}^\top (Y_n) \subseteq \mathscr{M}(w)$ by definition, we obtain the following by \cref{t:mitosis_formula}.

\begin{cor}
For $w \in W$, the equality $\mathscr{M}(w) = \mathscr{L}(D(w))$ holds.
\end{cor}

We also use the following fundamental properties (Lemmas \ref{l:shape_of_reduced_pipe} and \ref{l:inverse_ladder_move}) of reduced pipe dreams.

\begin{lem}[{see the proof of \cite[Lemma 3.6]{BB}}]\label{l:shape_of_reduced_pipe}
Let $w \in W$, and $D \in \mathscr{M}(w)$. 
If there exists $(i, j) \in D$ satisfying the following conditions:
\begin{itemize}
\item $(i, j+1) \notin D$,
\item there exists $1 \leq r < i$ such that $(i-r, j) \notin D$ and $(i-k, j), (i-k, j+1) \in D$ for all $1 \leq k < r$,
\end{itemize}
then it holds that $(i-r, j+1) \notin D$.
\end{lem}

\begin{lem}[{see the proof of \cite[Lemma 3.5]{BB}}]\label{l:inverse_ladder_move}
Let $w \in W$, and $D \in \mathscr{M}(w)$. 
If there exist $D^\prime \in \mathcal{PD}_n$ and $(i, j) \in Y_n$ such that $D = L_{i, j} D^\prime$, then $D^\prime \in \mathscr{M}(w)$. 
\end{lem}

The following is the main result of this section. 

\begin{thm}\label{t:main_result_type_A}
For $w \in W$, it holds that
\[\mathcal{C}_{{\bm i}_A}(X_w) = \bigcup_{D \in \mathscr{M}(w)} F_{{\bm k}_D}^\vee (\mathcal{C}_{{\bm i}_A}).\] 
\end{thm}

\begin{proof}
Since $\mathcal{C}_{{\bm i}_A}(X_w) \cap \z^N = \Phi_{{\bm i}_A}(\mathcal{B}_w(\infty))$, the assertion is equivalent to the equality
\[\Phi_{{\bm i}_A}(\mathcal{B}_w(\infty)) = \bigcup_{D \in \mathscr{M}(w)} F_{{\bm k}_D}^\vee (\mathcal{C}_{{\bm i}_A}) \cap \z^N.\]
We consider the Kashiwara embedding $\Psi_{\bm j} \colon \mathcal{B}(\infty) \hookrightarrow \z^\infty _{\bm j}$ associated with an infinite sequence ${\bm j} = (\ldots, j_2, j_1)$ extending ${\bm i}_A$ as in Sect.\ \ref{s:crystal_bases}. 
By \cref{p:star_string} (3), it suffices to prove that 
\[\Psi_{\bm j} (\mathcal{B}_{w^{-1}}(\infty)) = \bigcup_{D \in \mathscr{M}(w)} F_{{\bm k}_D}^\vee({\bm j}),\]
where 
\[F_{{\bm k}_D}^\vee({\bm j}) \coloneqq \{(\ldots, 0, 0, a_N, \ldots, a_2, a_1) \in \Psi_{\bm j} (\mathcal{B}(\infty)) \mid (a_1, \ldots, a_N) \in F_{{\bm k}_D}^\vee (\mathcal{C}_{{\bm i}_A})\}.\]

We proceed by induction on $\ell(w)$.
If $\ell(w) = 0$, then the assertion is obvious. 
If $\ell(w) = 1$, then we have $w = s_i$ for some $i \in I$.
By the definition of the crystal $\z^\infty _{\bm j}$, the image $\Psi_{\bm j} (\mathcal{B}_{s_i}(\infty)) = \{\tilde{f}_i^t \Psi_{\bm j} (b_\infty) \mid t \in \z_{\geq 0}\}$ coincides with the set of $(\ldots, a_l, \ldots, a_2, a_1) \in \z^\infty_{\bm j}$ such that $a_k = 0$ for all $k \neq \min\{l \in \z_{\geq 1} \mid j_l = i\}$.
Hence the assertion holds by \cref{ex:length_one_type_A}.

Assume that $\ell \coloneqq \ell(w) \geq 2$, and write ${\bm k}_{D(w)}^\prime = (k_1, \ldots, k_\ell)$. 
We set $i \coloneqq i_{k_\ell}$ and $w^\prime \coloneqq w s_i$. 
Then we have ${\bm k}_{D(w^\prime)}^\prime = (k_1, \ldots, k_{\ell-1})$.
By the induction hypothesis, it follows that 
\begin{equation}\label{eq:main_type_A_induction_hypothesis}
\begin{aligned}
\Psi_{\bm j} (\mathcal{B}_{(w^\prime)^{-1}}(\infty)) = \bigcup_{D \in \mathscr{M}(w^\prime)} F_{{\bm k}_D}^\vee({\bm j}).
\end{aligned}
\end{equation}

We take $b \in \mathcal{B}_{w^{-1}}(\infty)$.
Let us prove that $\Psi_{\bm j} (b) \in F_{{\bm k}_D}^\vee({\bm j})$ for some $D \in \mathscr{M}(w)$.
By \cref{p:property_of_Demazure_infty}, there exist $b^\prime \in \mathcal{B}_{(w^\prime)^{-1}}(\infty)$ and $s \in \z_{\geq 0}$ such that $b = \tilde{f}_i^s b^\prime$.
We set
\begin{align*}
&\Psi_{\bm j}(b) \eqqcolon {\bm a} = (\ldots, 0, 0, a_N, \ldots, a_2, a_1)\ {\rm and}\ \Psi_{\bm j}(b^\prime) \eqqcolon {\bm a}^\prime = (\ldots, 0, 0, a_N^\prime, \ldots, a_2^\prime, a_1^\prime).
\end{align*}
Then it follows that ${\bm a} = \Psi_{\bm j}(b) = \tilde{f}_i^s \Psi_{\bm j}(b^\prime) = \tilde{f}_i^s {\bm a}^\prime$.
We write $a_{(p_k, i_k)} \coloneqq a_k$, $a_{(p_k, i_k)}^\prime \coloneqq a_k^\prime$, and $\sigma_{(p_k, i_k)} ({\bm a}^\prime) \coloneqq \sigma_k ({\bm a}^\prime)$ for $1 \leq k \leq N$, where $\sigma_k ({\bm a}^\prime)$ is the notation used in the definition of $\z^\infty _{\bm j}$. 
For $(p, q) \in Y_n$, we have
\begin{equation}\label{eq:diagram_formula_of_sigma_for_crystal}
\begin{aligned}
\sigma_{(p, q)} ({\bm a}^\prime) = a_{(p, q)}^\prime - a_{(p, q-1)}^\prime + \sum_{1 \leq p^\prime \leq p-1} (-a_{(p^\prime, q-1)}^\prime +2a_{(p^\prime, q)}^\prime -a_{(p^\prime, q+1)}^\prime),
\end{aligned}
\end{equation}
where $a_{(p^\prime, q^\prime)}^\prime \coloneqq 0$ if $(p^\prime, q^\prime) \notin Y_n$.
By \eqref{eq:main_type_A_induction_hypothesis}, there exists $D^\prime \in \mathscr{M}(w^\prime)$ such that ${\bm a}^\prime \in F_{{\bm k}_{D^\prime}}^\vee({\bm j})$.
If we have $a_k = a_k^\prime$ for all $1 \leq k \leq N$ such that $(p_k, i_k) \in D^\prime$ and such that $k \neq k_\ell$, then ${\bm a} \in F_{{\bm k}_D}^\vee({\bm j})$ for $D \coloneqq D^\prime \setminus \{(p_{k_\ell}, i_{k_\ell} = i)\} \in \mathscr{M}(w)$. 
Assume that there exists $1 \leq k \leq N$ such that $(p_k, i_k) \in D^\prime$, $k \neq k_\ell$, and $a_k \neq a_k^\prime$. 
Take such $k$ arbitrary. 
Then the action of $\tilde{f}_i^s$ on ${\bm a}^\prime$ changes the coordinate $a_k^\prime$, which implies by the definition of $\z^\infty _{\bm j}$ that $i_k = i$ and that
\begin{equation}\label{eq:computation_of_sigma_for_k}
\begin{aligned}
\sigma_k ({\bm a}^\prime) > \max\{\sigma_y ({\bm a}^\prime) \mid y < k,\ i_y = i\}.
\end{aligned}
\end{equation}
Since ${\bm a}^\prime \in F_{{\bm k}_{D^\prime}}^\vee({\bm j})$, this implies by \eqref{eq:diagram_formula_of_sigma_for_crystal} that $(p_k, i + 1) \notin D^\prime$. 
Hence it follows by \cref{c:last_box_for_w_prime} that $p_k > p_{k_\ell}$. 
Since we have $(p_{k_\ell}, i + 1) \notin D^\prime$ by \cref{c:last_box_for_w_prime}, there exists $0 < t \leq p_k - p_{k_\ell}$ such that $(p_k - t, i + 1) \notin D^\prime$ and $(p_k-z, i+1) \in D^\prime$ for all $1 \leq z < t$. 
Then it holds by \cref{l:shape_of_reduced_pipe} that $(p_k - z, i) \in D^\prime$ for all $1 \leq z < t$. 
Hence we see by \eqref{eq:diagram_formula_of_sigma_for_crystal} that
\[\sigma_{(p_k, i)} ({\bm a}^\prime) = \sigma_{(p_k - 1, i)} ({\bm a}^\prime) = \cdots = \sigma_{(p_k - t + 1, i)} ({\bm a}^\prime),\]
which implies that 
\begin{align*}
\sigma_{(p_k - t, i)} ({\bm a}^\prime) - \sigma_{(p_k, i)} ({\bm a}^\prime) &= \sigma_{(p_k - t, i)} ({\bm a}^\prime) - \sigma_{(p_k - t + 1, i)} ({\bm a}^\prime)\\
&= a_{(p_k - t, i + 1)}^\prime - a_{(p_k - t, i)}^\prime \geq 0.
\end{align*}
Combining this with \eqref{eq:computation_of_sigma_for_k}, we deduce by the definition of $\z^\infty _{\bm j}$ that $a_{(p_k - t, i + 1)} = a_{(p_k - t, i)}$, that is, ${\bm a} \in F_{{\bm k}_{\{(p_k - t, i + 1)\}}}^\vee({\bm j})$.
If $(p_k-t, i) \notin D^\prime \setminus \{(p_{k_\ell}, i)\}$, then $L_{p_k, i} (D^\prime \setminus \{(p_{k_\ell}, i)\})$ is well-defined and coincides with $D^\prime \setminus \{(p_{k_\ell}, i), (p_k, i)\} \cup \{(p_k - t, i + 1)\}$. 
If $(p_k-t, i) \in D^\prime \setminus \{(p_{k_\ell}, i)\}$, then repeat this argument by replacing $(p_k, i)$ with $(p_k-t, i)$.
Then we obtain a sequence $(0 < t = t_1 < \cdots < t_q < t_{q+1} \leq p_k - p_{k_\ell})$ such that $L_{p_k, i} L_{p_k-t_1, i} \cdots L_{p_k-t_q, i}(D^\prime \setminus \{(p_{k_\ell}, i)\})$ is well-defined and coincides with 
\[D^\prime \setminus \{(p_{k_\ell}, i), (p_k, i), (p_k-t_1, i), \ldots, (p_k-t_q, i)\} \cup \{(p_k - t_1, i + 1), \ldots, (p_k - t_{q+1}, i + 1)\}.\]
In addition, it holds that ${\bm a} \in F_{{\bm k}_{\{(p_k - t_1, i + 1), \ldots, (p_k - t_{q+1}, i + 1)\}}}^\vee({\bm j})$.
Thus, we deduce that $\Psi_{\bm j}(b) = {\bm a} \in F_{{\bm k}_D}^\vee({\bm j})$ for some $D \in \mathcal{PD}_n$ which is obtained from $D^\prime \setminus \{(p_{k_\ell}, i)\}$ by applying a sequence of ladder moves of the form $L_{p, i}$. 
Since $D \in \mathscr{M}(w)$, we have proved that $\Psi_{\bm j} (\mathcal{B}_{w^{-1}}(\infty)) \subseteq \bigcup_{D \in \mathscr{M}(w)} F_{{\bm k}_D}^\vee({\bm j})$.

Conversely, let us prove that $\bigcup_{D \in \mathscr{M}(w)} F_{{\bm k}_D}^\vee({\bm j}) \subseteq \Psi_{\bm j} (\mathcal{B}_{w^{-1}}(\infty))$.
For each $D \in \mathscr{M}(w)$, we take $b \in \mathcal{B}(\infty)$ such that $\Psi_{\bm j}(b) \eqqcolon {\bm a} = (\ldots, 0, 0, a_N, \ldots, a_2, a_1)$ satisfies the following conditions:
\begin{itemize}
\item $a_{(p, q)}$ is sufficiently larger than $a_{(p, q-1)}$ if $(p, q) \in Y_n \setminus D$;
\item $a_{(p, q)}-a_{(p, q-1)}$ is sufficiently larger than $a_{(p^\prime, q^\prime)}-a_{(p^\prime, q^\prime-1)}$ for $(p, q), (p^\prime, q^\prime) \in Y_n \setminus D$ such that $p^\prime < p$,
\end{itemize}
where we write $a_{(p_k, i_k)} \coloneqq a_k$ for $1 \leq k \leq N$ and set $a_{(p, 0)} \coloneqq 0$.
Then we see that $(a_1, \ldots, a_N)$ is included in the relative interior of $F_{{\bm k}_D}^\vee (\mathcal{C}_{{\bm i}_A})$.
Let us prove that $b \in \mathcal{B}_{w^{-1}}(\infty)$.
Setting $b^\prime \coloneqq \tilde{e}_i^{\varepsilon_i(b)} b$, it suffices to show that $b^\prime \in \mathcal{B}_{(w^\prime)^{-1}}(\infty)$ by \cref{p:property_of_Demazure_infty}.
Write $\Psi_{\bm j}(b^\prime) \eqqcolon {\bm a}^\prime = (\ldots, 0, 0, a_N^\prime, \ldots, a_2^\prime, a_1^\prime)$, and set $a_{(p_k, i_k)}^\prime \coloneqq a_k^\prime$, $\sigma_{(p_k, i_k)} ({\bm a}) \coloneqq \sigma_k ({\bm a})$ for $1 \leq k \leq N$.
By \cref{l:inverse_ladder_move} and the definition of $\mathscr{M}(w)$, there exist $D^\prime \in \mathscr{M}(w^\prime)$ and $(u_1, i), \ldots, (u_r, i) \in Y_n$ such that $D = L_{u_r, i} \cdots L_{u_1, i} (D^\prime \setminus \{(p_{k_\ell}, i)\})$ and such that $D^\prime \setminus \{(p_{k_\ell}, i)\}$ cannot be written as $L_{u, i}(D'' \setminus \{(p_{k_\ell}, i)\})$ for some $u$ and $D'' \in \mathcal{PD}_n$.
We define $u_1^\prime, \ldots, u_r^\prime$ by 
\[L_{u_k, i} \cdots L_{u_1, i} (D^\prime \setminus \{(p_{k_\ell}, i)\}) = L_{u_{k-1}, i} \cdots L_{u_1, i} (D^\prime \setminus \{(p_{k_\ell}, i)\}) \setminus \{(u_k, i)\} \cup \{(u_k^\prime, i+1)\}\] 
for $1 \leq k \leq r$.
In a way similar to the proof of $\Psi_{\bm j} (\mathcal{B}_{w^{-1}}(\infty)) \subseteq \bigcup_{D \in \mathscr{M}(w)} F_{{\bm k}_D}^\vee({\bm j})$, it holds that
\begin{align*}
&\sigma_{(u_k^\prime, i)} ({\bm a}) = \sigma_{(u_k^\prime+1, i)} ({\bm a}) = \cdots = \sigma_{(u_k-1, i)} ({\bm a})\ {\rm and}\\
&\sigma_{(u_k^\prime, i)} ({\bm a}) - \sigma_{(u_k, i)} ({\bm a}) = a_{(u_k, i - 1)} - a_{(u_k, i)} \leq 0
\end{align*}
for $1 \leq k \leq r$. 
In addition, since $(u_k^\prime, i) \notin D$, we see by the choice of ${\bm a}$ that $\sigma_{(u, i)} ({\bm a}) < \sigma_{(u_k^\prime, i)} ({\bm a})$ for all $1 \leq u < u_k^\prime$.
We also have $\sigma_{(p_{k_\ell}, i)} ({\bm a}) = a_{(p_{k_\ell}, i)}$ and $\sigma_{(p, i)} ({\bm a}) = 0$ for all $1 \leq p < p_{k_\ell}$.
From these and the definition of $\z^\infty _{\bm j}$, it follows that $a_{(u_k, i - 1)}^\prime = a_{(u_k, i)}^\prime$ for all $1 \leq k \leq r$ and that $a_{(p_{k_\ell}, i)}^\prime = 0$, where $a_{(u_k, i - 1)}^\prime \coloneqq 0$ if $i = 1$.
Since $(n-i+1, i+1) \notin Y_n$, \cref{l:shape_of_reduced_pipe} and the choice of $D^\prime$ imply that if $(p, i) \in Y_n \setminus D$ satisfies $(p, i+1) \in D$, then $p = u_k^\prime$ for some $1 \leq k \leq r$.
Hence we have ${\bm a}^\prime \in F_{{\bm k}_{D^\prime}}^\vee({\bm j})$, which implies by \eqref{eq:main_type_A_induction_hypothesis} that $b^\prime \in \mathcal{B}_{(w^\prime)^{-1}}(\infty)$. 
Thus, we have proved that $b \in \mathcal{B}_{w^{-1}}(\infty)$. 
Since $(a_1, \ldots, a_N)$ is included in the relative interior of $F_{{\bm k}_D}^\vee (\mathcal{C}_{{\bm i}_A})$, this implies by \cref{c:union_of_faces_string_cone} that $F_{{\bm k}_D}^\vee (\mathcal{C}_{{\bm i}_A}) \subseteq \mathcal{C}_{{\bm i}_A}(X_w)$, which is equivalent to the inclusion $F_{{\bm k}_D}^\vee ({\bm j}) \subseteq \Psi_{\bm j} (\mathcal{B}_{w^{-1}}(\infty))$.
This proves the theorem.
\end{proof}

In a way similar to the proof of \cref{l:opposite_Demazure_union_of_faces}, we obtain the following by \cref{t:main_result_type_A}.

\begin{cor}\label{c:Kogan_face_string_polytope}
For $w \in W$ and $\lambda \in P_+$, the following equalities hold: 
\begin{align*}
\mathscr{C}_{{\bm i}_A}(X_w) = \bigcup_{D \in \mathscr{M}(w)} F_{{\bm k}_D}^\vee (\mathscr{C}_{{\bm i}_A})\quad{\it and}\quad \Delta_{{\bm i}_A}(\lambda, X_w) = \bigcup_{D \in \mathscr{M}(w)} F_{{\bm k}_D}^\vee (\Delta_{{\bm i}_A}(\lambda)).
\end{align*}
\end{cor}

The following is an immediate consequence of \cref{t:ladder_moves_BB} and \cref{c:Kogan_face_string_polytope}.

\begin{cor}\label{c:Kogan_Demazure_crystal}
For $w \in W$ and $\lambda \in P_+$, the union $\bigcup_{D \in \mathscr{M}(w)} F_{{\bm k}_D}^\vee(GT(\lambda))$ of reduced Kogan faces $F_{{\bm k}_D}^\vee(GT(\lambda))$ with $w(F_{{\bm k}_D}^\vee(GT(\lambda))) = w_0 w$ gives a polyhedral parametrization of $\mathcal{B}_w(\lambda)$ under the unimodular affine transformation $GT(\lambda) \simeq \Delta_{{\bm i}_A}(\lambda)$.
\end{cor}

\begin{rem}
Such relation between reduced Kogan faces and the character of $\mathcal{B}_w(\lambda)$ was previously given in \cite[Theorem 5.1]{KST} and in \cite[Lemma 4.7]{BF}.
\end{rem}

In addition, we obtain the following by \cref{t:semi-toric_degenerations} (3) and \cref{c:Kogan_face_string_polytope}.

\begin{cor}
For $w \in W$ and $\lambda \in P_{++}$, the Schubert variety $X_w$ degenerates into the union $\bigcup_{D \in \mathscr{M}(w)} X(F_{{\bm k}_D}^\vee (\Delta_{{\bm i}_A}(\lambda)))$ of irreducible normal toric subvarieties $X(F_{{\bm k}_D}^\vee (\Delta_{{\bm i}_A}(\lambda)))$ of $X(\Delta_{{\bm i}_A}(\lambda))$ corresponding to the faces $F_{{\bm k}_D}^\vee (\Delta_{{\bm i}_A}(\lambda))$, $D \in \mathscr{M}(w)$.
\end{cor}

In the rest of this section, we review Kiritchenko-Smirnov-Timorin's results \cite{KST} to make explanations in Sect.\ \ref{s:type_C} clear.
Let $\mathscr{P} = \mathscr{P}_N$ denote the set of convex polytopes in $\r^N$. 
This set is endowed with a commutative semigroup structure by the Minkowski sum of convex polytopes: 
\[Q_1 + Q_2 \coloneqq \{x + y \mid x \in Q_1,\ y \in Q_2\}.\] 
Since this semigroup has the cancellation property, we can embed $\mathscr{P}$ into its Grothendieck group $G(\mathscr{P})$; an element of $G(\mathscr{P})$ is called a \emph{virtual polytope}. 
For $c \in \r_{\ge 0}$ and a convex polytope $Q \in \mathscr{P}$, we define a convex polytope $c Q$ by $c Q \coloneqq \{c x \mid x \in Q\}$; this induces an $\r$-linear space structure on $G(\mathscr{P})$. 
It is easy to see that 
\begin{equation}\label{eq:additivity_of_GT}
\begin{aligned}
GT(\lambda + \mu) = GT(\lambda) + GT(\mu)\quad{\rm for\ all}\ \lambda, \mu \in P_+,
\end{aligned}
\end{equation}
and that the Gelfand--Tsetlin polytopes $GT(\lambda)$, $\lambda \in P_{++}$, have the same normal fan which we denote by $\Sigma_A$.
The corresponding toric variety $X(\Sigma_A)$ is singular in general.
However, it has a resolution of singularities whose exceptional locus does not contain any divisors, which is called a \emph{small resolution}. 
More strongly, Nishinou--Nohara--Ueda \cite{NNU} proved the following.

\begin{prop}[{\cite[Proposition 4.1]{NNU}}]\label{p:small_resolution_type_A}
Every refinement of the fan $\Sigma_A$ into simplicial cones without adding new rays gives a small resolution of $X(\Sigma_A)$. 
\end{prop}

\begin{rem}
For other string polytopes of type $A_n$, Cho--Kim--Lee--Park \cite{CKLP} studied when we have small resolutions.
\end{rem}

For $\lambda \in P_+$ and ${\bm \epsilon} = (\epsilon_1, \ldots, \epsilon_n) \in \z^n$ such that $0 = \epsilon_1 \leq \epsilon_2 \leq \cdots \leq \epsilon_n$, we define an integral convex polytope $\widetilde{GT}_{\bm \epsilon}(\lambda)$ to be the set of \eqref{eq:coordinate_type_A_string} satisfying the following inequalities:
\begin{align*}
&a_j ^{(i)} + \epsilon_{i+1} \geq a_j ^{(i+1)} \geq a_{j+1} ^{(i)}\quad {\rm for}\ 0 \leq i \leq n-1\ {\rm and}\ 1 \leq j \leq n-i.
\end{align*}
If $\epsilon_1 = \cdots = \epsilon_n$, then we have $\widetilde{GT}_{\bm \epsilon}(\lambda) = GT(\lambda)$. 
We fix ${\bm \epsilon} = (\epsilon_1, \ldots, \epsilon_n) \in \z^n$ such that $0 = \epsilon_1 < \epsilon_2 < \cdots < \epsilon_n$. 
Then, for $\lambda \in P_{++}$, the polytope $\widetilde{GT}_{\bm \epsilon}(\lambda)$ is simple, and its normal fan $\widetilde{\Sigma}_A$ gives an example of \cref{p:small_resolution_type_A}.
Note that $\widetilde{\Sigma}_A$ does not depend on the choices of ${\bm \epsilon}$ and $\lambda$.
By \eqref{eq:additivity_of_GT}, the map $P_+ \rightarrow \mathscr{P}$, $\lambda \mapsto GT(\lambda)$, induces an injective group homomorphism $P \hookrightarrow G(\mathscr{P})$. 
Denote by $\Lambda_{GT} \subseteq G(\mathscr{P})$ the image of $P$ in $G(\mathscr{P})$, by $\Lambda_{\widetilde{GT}} \subseteq G(\mathscr{P})$ the $\z$-submodule generated by \emph{all} integral convex polytopes whose normal fans coincide with $\widetilde{\Sigma}_A$, and by $V_{\widetilde{GT}} \subseteq G(\mathscr{P})$ the $\r$-linear subspace spanned by $\Lambda_{\widetilde{GT}}$. 
Fix $\lambda \in P_{++}$. 
Let $\Gamma_1, \ldots, \Gamma_z$ be the facets of $\widetilde{GT}_{\bm \epsilon}(\lambda)$, and $\xi_1, \ldots, \xi_z$ the primitive vectors in the corresponding rays of $\widetilde{\Sigma}_A$. 
For $1 \leq i \leq z$, define an $\r$-linear function $x_i \colon V_{\widetilde{GT}} \rightarrow \r$ by 
\[x_i (Q) \coloneqq \max \{\langle \xi_i, p \rangle \mid p \in Q\}\]
for each convex polytope $Q \in V_{\widetilde{GT}}$. 
Then the $\r$-linear space $V_{\widetilde{GT}}$ is identified with $\r^z$ by arranging the values of $x_1, \ldots, x_z$ (see the proof of \cite[Theorem 2.1.1]{Tim}). 
In addition, by \cite[Theorem 2.1.1]{Tim}, there exists a homogeneous polynomial ${\rm vol}_{\widetilde{GT}}$ of degree $N$ on $V_{\widetilde{GT}}$, called the \emph{volume polynomial} on $V_{\widetilde{GT}}$, such that the value ${\rm vol}_{\widetilde{GT}}(Q)$ for each convex polytope $Q \in V_{\widetilde{GT}}$ equals the $N$-dimensional volume of $Q$.
Note that the $\z$-module $\Lambda_{\widetilde{GT}}$ is a lattice of rank $z$, and we have $\Lambda_{\widetilde{GT}}/\z^N \simeq {\rm Pic}(X(\widetilde{\Sigma}_A))$, where $X(\widetilde{\Sigma}_A)$ is the normal toric variety corresponding to the fan $\widetilde{\Sigma}_A$, and the right action of $\z^N$ on $\Lambda_{\widetilde{GT}}$ is given by translations of convex polytopes $Q \subseteq \r^N$ (see \cite[Theorem 4.2.1]{CLS}).
We define a $\z$-algebra $R_{\widetilde{GT}}$, called a \emph{polytope ring}, by 
\[R_{\widetilde{GT}} \coloneqq {\rm Sym}(\Lambda_{\widetilde{GT}})/J_{\widetilde{GT}},\] 
where $J_{\widetilde{GT}}$ is the homogeneous ideal of ${\rm Sym}(\Lambda_{\widetilde{GT}})$ given by  
\[J_{\widetilde{GT}} \coloneqq \{D \in {\rm Sym}(\Lambda_{\widetilde{GT}}) \mid D \cdot {\rm vol}_{\widetilde{GT}} = 0\},\]
where ${\rm Sym}(\Lambda_{\widetilde{GT}})$ is regarded as a ring of differential operators on $\r[V_{\widetilde{GT}}]$ in a way similar to ${\rm Sym}(P_\r)$ in \cref{t:Borel_description_polytope}.
By \cite[Sect.\ 1.4]{KhPu}, this $\z$-algebra $R_{\widetilde{GT}}$ is isomorphic to the cohomology ring $H^\ast(X(\widetilde{\Sigma}_A); \z)$.
Similarly, the volume polynomial ${\rm vol}_{GT}$ on $V_{GT} \coloneqq \Lambda_{GT} \otimes_\z \r$ and the polytope ring $R_{GT} = {\rm Sym}(\Lambda_{GT})/J_{GT}$ are defined. 
Then it follows by \cref{t:Borel_description_polytope} and \cref{r:Borel_description_over_Z} that the polytope ring $R_{GT}$ is isomorphic to the cohomology ring $H^\ast(G/B; \z)$. 
Since we have 
\[\widetilde{GT}_{\bm \epsilon}(\lambda) + GT(\mu) = \widetilde{GT}_{\bm \epsilon}(\lambda + \mu)\] 
for $\mu \in P_+$, it follows that $GT(\mu) \in \Lambda_{\widetilde{GT}}$. 
Hence we see that $\Lambda_{GT} \subseteq \Lambda_{\widetilde{GT}}$. 
Let $\gamma \colon \Lambda_{GT} \hookrightarrow \Lambda_{\widetilde{GT}}$ denote the inclusion map, and $\gamma_\ast \colon {\rm Sym}(\Lambda_{GT}) \rightarrow {\rm Sym}(\Lambda_{\widetilde{GT}})$ the $\z$-algebra homomorphism induced by $\gamma$. 

\begin{prop}[{see \cite[Proposition 2.1]{KST}}]
There exist a $\z$-module $M_{\widetilde{GT}, GT}$, a surjective $\z$-module homomorphism $\pi \colon R_{\widetilde{GT}} \twoheadrightarrow M_{\widetilde{GT}, GT}$, and an injective $\z$-module homomorphism $\iota \colon R_{GT} \hookrightarrow M_{\widetilde{GT}, GT}$ such that 
\begin{itemize}
\item if $\pi(f^\prime) = \iota(f)$ and $\pi(g^\prime) = \iota(g)$ for some $f^\prime, g^\prime \in R_{\widetilde{GT}}$ and $f, g \in R_{GT}$, then it holds that $\pi(f^\prime g^\prime) = \iota(f g)$,
\item for all $D \in {\rm Sym}(\Lambda_{GT})$, it holds that
\[\iota(D \bmod J_{GT}) = \pi(\gamma_\ast (D) \bmod J_{\widetilde{GT}}).\]
\end{itemize}
\end{prop}

For $1 \leq i \leq z$, we define a derivation $\partial_i \in \Lambda_{\widetilde{GT}}$ on $\r[V_{\widetilde{GT}}]$ by $\partial_i \coloneqq \frac{\partial}{\partial x_i}$. 
For each face $F = \Gamma_{k_1} \cap \cdots \cap \Gamma_{k_\ell}$ of $\widetilde{GT}_{\bm \epsilon}(\lambda)$ of codimension $\ell$, we write 
\[\partial_F \coloneqq \partial_{k_1} \cdots \partial_{k_\ell} \in {\rm Sym}(\Lambda_{\widetilde{GT}}),\]
and denote by $[F]$ the corresponding element of $R_{\widetilde{GT}}$.
By abuse of notation, the image $\pi([F]) \in M_{\widetilde{GT}, GT}$ is also denoted by $[F]$.
The elements $[F]$ for all faces $F$ of $\widetilde{GT}_{\bm \epsilon}(\lambda)$ generate the polytope ring $R_{\widetilde{GT}}$ as a $\z$-module by \cite[Lemma 2.6.3]{Tim}, which implies that their images in $M_{\widetilde{GT}, GT}$ give a $\z$-module generator of $M_{\widetilde{GT}, GT}$.  
For a Kogan face $F_{\bm k}^\vee(GT(\lambda))$ (resp., a dual Kogan face $F_{\bm k}(GT(\lambda))$), there uniquely exists a face $F_{\bm k}^\vee(\widetilde{GT}_{\bm \epsilon}(\lambda))$ (resp., $F_{\bm k}(\widetilde{GT}_{\bm \epsilon}(\lambda))$) of $\widetilde{GT}_{\bm \epsilon}(\lambda)$ such that the corresponding cones $C_{F_{\bm k}^\vee(GT(\lambda))} \in \Sigma_A$ and $C_{F_{\bm k}^\vee(\widetilde{GT}_{\bm \epsilon}(\lambda))} \in \widetilde{\Sigma}_A$ (resp., $C_{F_{\bm k}(GT(\lambda))} \in \Sigma_A$ and $C_{F_{\bm k}(\widetilde{GT}_{\bm \epsilon}(\lambda))} \in \widetilde{\Sigma}_A$) coincide (see \cite[Proposition 2.3]{KST}).
Then we define $[F_{\bm k}^\vee(GT)], [F_{\bm k}(GT)] \in M_{\widetilde{GT}, GT}$ by 
\[
[F_{\bm k}^\vee(GT)] \coloneqq [F_{\bm k}^\vee(\widetilde{GT}_{\bm \epsilon}(\lambda))]\quad {\rm and}\quad [F_{\bm k}(GT)] \coloneqq [F_{\bm k}(\widetilde{GT}_{\bm \epsilon}(\lambda))],
\]
which are independent of the choices of ${\bm \epsilon}$ and $\lambda$.

\begin{thm}[{\cite[Theorem 4.3 and Corollary 4.5]{KST}}]
For $w \in W$, the cohomology class $[X^w] = [X_{w_0 w}] \in H^\ast(G/B; \z) \simeq R_{GT}$ is described as follows: 
\begin{align*}
[X^w] = \sum_{{\bm k} \in R({\bm i}_A, w_0 w w_0)} [F_{\bm k}^\vee(GT)] = \sum_{{\bm k} \in R({\bm i}_A, w)} [F_{\bm k}(GT)].
\end{align*}
\end{thm}

\section{Symplectic Kogan faces and Demazure crystals of type $C$}\label{s:type_C}

In this section, we consider the case $G = Sp_{2n} (\c)$, and extend results in Sect.\ \ref{s:type_A} to this case. 
We identify the set $I$ of vertices of the Dynkin diagram with $\{1, \ldots, n\}$ as follows:
\begin{align*}
&C_n\ \begin{xy}
\ar@{=>} (50,0) *++!D{1} *\cir<3pt>{};
(60,0) *++!D{2} *\cir<3pt>{}="C"
\ar@{-} "C";(65,0) \ar@{.} (65,0);(70,0)^*!U{}
\ar@{-} (70,0);(75,0) *++!D{n-1} *\cir<3pt>{}="D"
\ar@{-} "D";(85,0) *++!D{n} *\cir<3pt>{}="E"
\end{xy}.
\end{align*}
Write $N \coloneqq n^2$, and define ${\bm i}_C = (i_1, \ldots, i_N) \in R(w_0)$ by
\begin{equation}\label{eq:reduced_word_type_C}
\begin{aligned}
{\bm i}_C \coloneqq (1, \underbrace{2, 1, 2}_3, \underbrace{3, 2, 1, 2, 3}_5, \ldots, \underbrace{n, n-1, \ldots, 1, \ldots, n-1, n}_{2n-1}).
\end{aligned}
\end{equation}
Then Littelmann \cite[Theorem 6.1]{Lit} proved that the string cone $\mathcal{C}_{{\bm i}_C}$ coincides with the set of 
\begin{equation}\label{eq:coordinate_type_C_string}
\begin{aligned}
(a_1 ^{(1)}, \underbrace{b_1 ^{(2)}, a_2 ^{(1)}, a_1 ^{(2)}}_3, \underbrace{b_1 ^{(3)}, b_2 ^{(2)},  a_3 ^{(1)},  a_2 ^{(2)},  a_1 ^{(3)}}_5, \ldots, \underbrace{b_1 ^{(n)}, \ldots, b_{n-1} ^{(2)}, a_n ^{(1)}, \ldots, a_1 ^{(n)}}_{2n-1}) \in \r^N
\end{aligned}
\end{equation}
satisfying the following inequalities:
\begin{align*}
&a_1 ^{(1)} \geq 0, & &b_1 ^{(2)} \geq a_2 ^{(1)} \geq a_1 ^{(2)} \geq 0, & &\ldots, & &b_1 ^{(n)} \geq \cdots \geq b_{n-1} ^{(2)} \geq a_n ^{(1)} \geq \cdots \geq a_1 ^{(n)} \geq 0.
\end{align*}
We arrange the equations of the facets of $\mathcal{C}_{{\bm i}_C}$ as 
\begin{align*}
a_1 ^{(1)} = 0,\ b_1 ^{(2)} = a_2 ^{(1)},\ a_2 ^{(1)} = a_1 ^{(2)},\ a_1 ^{(2)} = 0,\ b_1 ^{(3)} = b_2 ^{(2)},\ \ldots,\ a_2 ^{(n-1)} = a_1 ^{(n)},\ a_1 ^{(n)} = 0,
\end{align*}
and denote the corresponding facets of $\mathcal{C}_{{\bm i}_C}$ by $F_1^\vee(\mathcal{C}_{{\bm i}_C}), \ldots, F_N^\vee(\mathcal{C}_{{\bm i}_C})$, respectively. 
For $1 \leq j \leq N$, $\lambda \in P_+$, and ${\bm k} = (k_1, \ldots, k_\ell)$ with $\ell \geq 0$ and with $1 \leq k_1 < \cdots < k_\ell \leq N$, define $F_j^\vee(\Delta_{{\bm i}_C}(\lambda)), F_j^\vee(\mathscr{C}_{{\bm i}_C}), F_{\bm k}^\vee(\mathcal{C}_{{\bm i}_C}), F_{\bm k}^\vee(\Delta_{{\bm i}_C}(\lambda))$, and $F_{\bm k}^\vee(\mathscr{C}_{{\bm i}_C})$ as in Sect.\ \ref{s:type_A}.

For $\lambda \in P_+$ and $1 \leq k \leq n$, we set $b_k ^{(1)} \coloneqq \sum_{1 \leq \ell \leq n-k+1} \langle \lambda, h_\ell \rangle$, and write $b_{n+1}^{(1)} = b_n^{(2)} = \cdots = b_2^{(n)} \coloneqq 0$. 
Then the \emph{symplectic Gelfand--Tsetlin polytope} $SGT(\lambda)$ is defined to be the set of \eqref{eq:coordinate_type_C_string} satisfying the following inequalities:
\[\begin{matrix}
b_1 ^{(1)} & & b_2 ^{(1)} & & \cdots & b_n ^{(1)} & & b_{n+1} ^{(1)} \\
 & a_1 ^{(1)} & & a_2 ^{(1)} & & & a_n ^{(1)} & \\
 & & b_1 ^{(2)} & & \cdots & b_{n -1} ^{(2)}  & & b_n ^{(2)}\\
 & & & a_1 ^{(2)} & & & a_{n -1} ^{(2)} & \\
 & & & & \ddots & & & \vdots \\
 & & & & & b_1 ^{(n)} & & b_2 ^{(n)}\\
 & & & & & & a_1 ^{(n)}. &
\end{matrix}\]
The symplectic Gelfand--Tsetlin polytope $SGT(\lambda)$ is an integral polytope for all $\lambda \in P_+$, and it is $N$-dimensional if $\lambda \in P_{++}$.
A face of $SGT(\lambda)$ given by equations of the types $a_k^{(l)} = b_{k-1}^{(l+1)}$ and $b_k^{(l)} = a_k^{(l)}$ (resp., $a_k^{(l)} = b_k^{(l+1)}$ and $b_k^{(l)} = a_{k-1}^{(l)}$) is called a \emph{symplectic Kogan face} (resp., a \emph{dual symplectic Kogan face}).
We arrange the equations of the types $a_k^{(l)} = b_k^{(l+1)}$ and $b_k^{(l)} = a_{k-1}^{(l)}$ as
\begin{align*}
&b_2^{(n)} = a_1 ^{(n)},\ a_1 ^{(n-1)} = b_1^{(n)},\ a_2 ^{(n-1)} = b_2^{(n)},\ b_2^{(n-1)} = a_1 ^{(n-1)},\ a_1 ^{(n-2)} = b_1^{(n-1)},\ a_2 ^{(n-2)} = b_2^{(n-1)},\\ 
&a_3 ^{(n-2)} = b_3^{(n-1)},\ b_3^{(n-2)} = a_2 ^{(n-2)},\ \ldots,\ a_{n-1}^{(1)} = b_{n-1}^{(2)},\ a_n ^{(1)} = b_n^{(2)},\ b_n ^{(1)} = a_{n-1} ^{(1)},\ \ldots,\ b_2^{(1)} = a_1^{(1)},
\end{align*}
and denote the corresponding facets of $SGT(\lambda)$ by $F_1 (SGT(\lambda)), \ldots, F_N (SGT(\lambda))$, respectively.
We also arrange the equations of the types $a_k^{(l)} = b_{k-1}^{(l+1)}$ and $b_k^{(l)} = a_k^{(l)}$ as 
\begin{align*}
&b_1 ^{(n)} = a_1 ^{(n)},\ a_2 ^{(n-1)} = b_1 ^{(n)},\ b_2 ^{(n-1)} = a_2 ^{(n-1)},\ b_1 ^{(n-1)} = a_1 ^{(n-1)},\ a_2 ^{(n-2)} = b_1 ^{(n-1)},\ a_3 ^{(n-2)} = b_2 ^{(n-1)},\\ 
&b_3 ^{(n-2)} = a_3 ^{(n-2)},\ \ldots,\ a_{n-1} ^{(1)} = b_{n-2} ^{(2)},\ a_n ^{(1)} = b_{n-1} ^{(2)},\ b_n ^{(1)} = a_n ^{(1)},\ \ldots, b_2 ^{(1)} = a_2 ^{(1)},\ b_1 ^{(1)} = a_1 ^{(1)},
\end{align*}
and denote the corresponding facets of $SGT(\lambda)$ by $F^\vee_1(SGT(\lambda)), \ldots, F^\vee_N(SGT(\lambda))$, respectively.
For ${\bm k} = (k_1, \ldots, k_\ell)$ with $\ell \geq 0$ and with $1 \leq k_1 < \cdots < k_\ell \leq N$, define $F_{\bm k}(SGT(\lambda))$ (resp., $F_{\bm k}^\vee(SGT(\lambda))$) in a way similar to $F_{\bm k}(\Delta_{{\bm i}_C}(\lambda))$ (resp., $F_{\bm k}^\vee(\Delta_{{\bm i}_C}(\lambda))$).
Littelmann \cite[Corollary 7]{Lit} gave a unimodular affine transformation from the string polytope $\Delta_{{\bm i}_C}(\lambda)$ to $SGT(\lambda)$. 
Under this transformation, the face $F_{\bm k}(\Delta_{{\bm i}_C}(\lambda))$ (resp., $F_{\bm k}^\vee(\Delta_{{\bm i}_C}(\lambda))$) of $\Delta_{{\bm i}_C}(\lambda)$ corresponds to the dual symplectic Kogan face $F_{\bm k}(SGT(\lambda))$ (resp., the symplectic Kogan face $F_{\bm k}^\vee(SGT(\lambda))$) of $SGT(\lambda)$. 
Hence Corollaries \ref{c:string_parametrization_opposite_Demazure} and \ref{c:string_opposite_Demazure_volume} imply the following.

\begin{cor}\label{c:dual_Kogan_opposite_Demazure_type_C}
For $w \in W$ and $\lambda \in P_+$, the union $\bigcup_{{\bm k} \in R({\bm i}_C, w)} F_{\bm k}(SGT(\lambda))$ gives a polyhedral parametrization of $\mathcal{B}^w(\lambda)$ under the unimodular affine transformation $SGT(\lambda) \simeq \Delta_{{\bm i}_C}(\lambda)$.
In particular, it holds for $\lambda \in P_{++}$ that
\[{\rm Vol}(X^w, \mathcal{L}_\lambda) = \sum_{{\bm k} \in R({\bm i}_C, w)} {\rm Vol}_{N-\ell(w)} (F_{\bm k} (SGT(\lambda))).\]
\end{cor}

Let $\mathcal{SPD}_n$ denote the set of subsets of 
\[SY_n \coloneqq \{(i, j) \in \z^2 \mid 1 \leq i \leq n,\ i \leq j \leq 2n-i\}.\]
We regard $SY_n$ as a shifted Young diagram, and represent an element $D \in \mathcal{SPD}_n$ by putting $+$ in the boxes contained in $D$. 
For instance, a set $D = \{(1, 1), (1, 2), (1, 3), (2, 2), (2, 3)\} \in \mathcal{SPD}_3$ is represented as 
\[D = \begin{ytableau}
+ & + & + & \mbox{} & \mbox{} \\
\none & + & + & \mbox{} & \none \\
\none & \none & \mbox{} & \none & \none
\end{ytableau}.\]
An element of $\mathcal{SPD}_n$ is called a \emph{skew pipe dream} (see \cite[Sect.\ 5.2]{Kir}).
We arrange the elements of $SY_n$ as 
\[((p_1, q_1), \ldots, (p_N, q_N)) \coloneqq ((n, n), (n-1, n+1), (n-1, n), (n-1, n-1), \ldots, (1, 2), (1, 1)).\]
Then it holds that $q_k \in \{n-i_k+1, n+i_k-1\}$ for $1 \leq k \leq N$.
We associate to each $D \in \mathcal{SPD}_n$ a sequence ${\bm k}_D = (k_1, \ldots, k_{|D|})$ with $1 \leq k_1 < \cdots < k_{|D|} \leq N$ by arranging $1 \leq k \leq N$ such that $(p_k, q_k) \in D$.
For instance, a skew pipe dream $D = \{(1, 1), (1, 2), (1, 3), (2, 2), (2, 3)\} \in \mathcal{SPD}_3$ gives a sequence ${\bm k}_D = (3, 4, 7, 8, 9)$.
Let us introduce the notion of \emph{ladder moves} for skew pipe dreams, which is an analog of that for pipe dreams.
For $(i, j) \in SY_n$, the \emph{ladder move} $L_{i, j}$ is defined as follows: write $(i, j) = (p_k, q_k)$, and take $D \in \mathcal{SPD}_n$ satisfying the following conditions:
\begin{itemize}
\item $(i, j) \in D$, $(i, j+1) \notin D$,
\item there exists $k < \ell \leq N$ such that $q_\ell \in \{j, 2n - j\}$, $(p_\ell, q_\ell), (p_\ell, q_\ell + 1) \notin D$, and $(p_r, q_r), (p_r, q_r + 1) \in D$ for all $k < r < \ell$ such that $q_r \in \{j, 2n - j\}$.
\end{itemize}
Then we define $L_{i, j} (D) \in \mathcal{SPD}_n$ by 
\[L_{i, j} (D) \coloneqq D \cup \{(p_\ell, q_\ell + 1)\} \setminus \{(i, j)\}.\]
Denote the skew pipe dream $L_{i, j} (D)$ also by $L_{p_\ell, q_\ell}^\prime(D)$.

\begin{ex}\label{ex:ladder_moves_type_C}
Let $n = 3$. Then we obtain the following: 
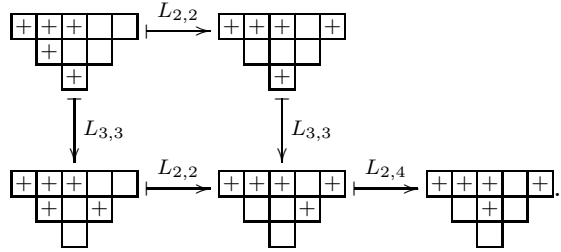
\begin{align*}
\xymatrix{
{\begin{ytableau}
+ & + & + & \mbox{} & \mbox{} \\
\none & + & \mbox{} & \mbox{} & \none \\
\none & \none & + & \none & \none
\end{ytableau}} \ar@{|->}[d]^-{L_{3, 3}} \ar@{|->}[r]^-{L_{2, 2}} & {\begin{ytableau}
+ & + & + & \mbox{} & + \\
\none & \mbox{} & \mbox{} & \mbox{} & \none \\
\none & \none & + & \none & \none
\end{ytableau}} \ar@{|->}[d]^-{L_{3, 3}} & \\
{\begin{ytableau}
+ & + & + & \mbox{} & \mbox{} \\
\none & + & \mbox{} & + & \none \\
\none & \none & \mbox{} & \none & \none
\end{ytableau}} \ar@{|->}[r]^-{L_{2, 2}} & {\begin{ytableau}
+ & + & + & \mbox{} & + \\
\none & \mbox{} & \mbox{} & + & \none \\
\none & \none & \mbox{} & \none & \none
\end{ytableau}} \ar@{|->}[r]^-{L_{2, 4}} & {\begin{ytableau}
+ & + & + & \mbox{} & + \\
\none & \mbox{} & + & \mbox{} & \none \\
\none & \none & \mbox{} & \none & \none
\end{ytableau}.} 
}
\end{align*}
\end{ex}

We associate to each $D \in \mathcal{SPD}_n$ a sequence ${\bm k}_D^\prime = (k_1, \ldots, k_{N-|D|})$ with $1 \leq k_1 < \cdots < k_{N-|D|} \leq N$ by arranging $1 \leq k \leq N$ such that $(p_k, q_k) \in SY_n \setminus D$.
For instance, a skew pipe dream $D = \{(1, 1), (1, 2), (1, 3), (2, 2), (2, 3)\} \in \mathcal{SPD}_3$ corresponds to a sequence ${\bm k}_D^\prime = (1, 2, 5, 6)$.
For $w \in W$, define $D(w) \in \mathcal{SPD}_n$ by the condition that ${\bm k}_{D(w)}^\prime$ is the minimum element in $R({\bm i}_C, w)$ with respect to the lexicographic order. 
If $w = e$, then we think of $D(w)$ as $SY_n$.
In a way similar to the proofs of \cref{p:relation_with_BB} and \cref{c:last_box_for_w_prime}, we obtain the following.

\begin{prop}\label{p:BB_type_description}
For $w \in W$, there exist $m(1), \ldots, m(n) \in \z_{\geq 0}$ such that 
\[D(w) = \{(i, j) \in SY_n \mid j \leq m(i) + i - 1\}.\]
In addition, writing ${\bm k}_{D(w)}^\prime = (k_1, \ldots, k_\ell)$ and $w^\prime \coloneqq w s_{i_{k_\ell}}$, it holds that $(p_{k_\ell}, q_{k_\ell}+1) \notin D$ and 
\[\{(i, j) \mid 1 \leq i \leq p_{k_\ell} - 1,\ i \leq j \leq 2n-i\} \cup \{(p_{k_\ell}, j) \mid p_{k_\ell} \leq j \leq q_{k_\ell}\} \subseteq D\]
for all $D \in \mathcal{SPD}_n$ which is obtained from $D(w^\prime)$ by applying a sequence of ladder moves.
\end{prop}

For $1 \leq i \leq n$, let us define an operator $M_i$ for skew pipe dreams, which is an analog of the operator $M_i$ for pipe dreams studied in Sect.\ \ref{s:type_A}.
For $D \in \mathcal{SPD}_n$, we set 
\[r_0 \coloneqq \max\{1\leq r \leq N \mid q_r \in \{n-i+1, n+i-1\},\ (p_r, q_r+1) \notin D\},\]
and assume that 
\begin{align*}
&\{(p_r, q_r) \mid q_r \in \{n-i+1, n+i-1\},\ r_0 \leq r \leq N\}\\ 
&\cup \{(p_r, q_r+1) \mid q_r \in \{n-i+1, n+i-1\},\ r_0+1 \leq r \leq N\} \subseteq D.
\end{align*}
Then the operator $M_i$ sends $D$ to the set $M_i(D)$ of elements of $\mathcal{SPD}_n$ obtained from $D \setminus \{(p_{r_0}, q_{r_0})\}$ by applying sequences of ladder moves $L_{p, q}$ such that $q \in \{n-i+1, n+i-1\}$.
For a subset $\mathcal{A} \subseteq \mathcal{SPD}_n$ such that $M_i(D)$ is defined for all $D \in \mathcal{A}$, we set $M_i(\mathcal{A}) \coloneqq \bigcup_{D \in \mathcal{A}} M_i(D)$.
For $w \in W$, we define a set $\mathscr{M}(w)$ of skew pipe dreams by 
\[\mathscr{M}(w) \coloneqq M_{i_{k_\ell}} \cdots M_{i_{k_1}} (SY_n),\]
where we write ${\bm k}^\prime_{D(w)} = (k_1, \ldots, k_\ell)$.
Note that $M_{i_{k_r}} \cdots M_{i_{k_1}} (SY_n) = \mathscr{M}(s_{i_{k_1}} \cdots s_{i_{k_r}})$ for each $1 \leq r \leq \ell$.
By \cref{p:BB_type_description}, the set $\mathscr{M}(w)$ is well-defined and coincides with the set of skew pipe dreams of the form $\widetilde{L}_\ell \widetilde{L}_{\ell-1} \cdots \widetilde{L}_2 D(w)$, where $\widetilde{L}_r$ for $2 \leq r \leq \ell$ is given by a sequence of ladder moves $L_{p_t, q_t}$ such that $t < k_r$ and such that $q_t \in \{n-i_{k_r}+1, n+i_{k_r}-1\}$.

\begin{rem}
The notion of mitosis operators for skew pipe dreams was introduced by Kiritchenko \cite[Sect.\ 5.2]{Kir}, but we do not use it in this paper. 
\end{rem}

\begin{ex}\label{ex:length_one_type_C}
For $1 \leq i \leq n$, we have $D(s_i) = SY_n \setminus \{(n-i+1, n+i-1)\}$ and $\mathscr{M}(s_i) = \{D(s_i)\}$. 
\end{ex}

\begin{ex}
Let $n = 2$. Then we see that 
\begin{align*}
&\mathscr{M}(s_1) = \left\{\begin{ytableau}
+ & + & +\\
\none & \mbox{} & \none 
\end{ytableau}\right\}, & &\mathscr{M}(s_2) = \left\{\begin{ytableau}
+ & + & \mbox{} \\
\none & + & \none 
\end{ytableau}\right\},\\
&\mathscr{M}(s_1 s_2) = \left\{\begin{ytableau}
+ & + & \mbox{}\\
\none & \mbox{} & \none 
\end{ytableau}\right\}, & &\mathscr{M}(s_2 s_1) = \left\{\begin{ytableau}
+ & \mbox{} & \mbox{}\\
\none & + & \none 
\end{ytableau}, \begin{ytableau}
+ & \mbox{} & + \\
\none & \mbox{} & \none 
\end{ytableau}\right\},\\
&\mathscr{M}(s_1 s_2 s_1) = \left\{\begin{ytableau}
+ & \mbox{} & \mbox{}\\
\none & \mbox{} & \none 
\end{ytableau}\right\}, & &\mathscr{M}(s_2 s_1 s_2) = \left\{\begin{ytableau}
\mbox{} & \mbox{} & \mbox{}\\
\none & + & \none 
\end{ytableau}, \begin{ytableau}
\mbox{} & \mbox{} & +\\
\none & \mbox{} & \none 
\end{ytableau}, \begin{ytableau}
\mbox{} & + & \mbox{}\\
\none & \mbox{} & \none 
\end{ytableau}\right\}.
\end{align*}
These sets of skew pipe dreams coincide with those obtained by Kiritchenko \cite[Example 2.9]{Kir} using mitosis operators.
\end{ex}

\begin{ex}
Let $n = 3$, and $w = s_2 s_1 s_3 s_2$. Then it follows that 
\[D(w) = \begin{ytableau}
+ & + & + & \mbox{} & \mbox{} \\
\none & + & \mbox{} & \mbox{} & \none \\
\none & \none & + & \none & \none
\end{ytableau}.\]
Hence the set $\mathscr{M}(w)$ consists of the five skew pipe dreams given in \cref{ex:ladder_moves_type_C}.
\end{ex}

Now we obtain analogs of \cref{t:main_result_type_A} and \cref{c:Kogan_face_string_polytope} as follows.

\begin{thm}\label{t:main_result_string_Demazure_type_C}
For $w \in W$ and $\lambda \in P_+$, the equalities $\mathcal{C}_{{\bm i}_C}(X_w) = \bigcup_{D \in \mathscr{M}(w)} F_{{\bm k}_D}^\vee (\mathcal{C}_{{\bm i}_C})$, $\mathscr{C}_{{\bm i}_C}(X_w) = \bigcup_{D \in \mathscr{M}(w)} F_{{\bm k}_D}^\vee (\mathscr{C}_{{\bm i}_C})$, and $\Delta_{{\bm i}_C}(\lambda, X_w) = \bigcup_{D \in \mathscr{M}(w)} F_{{\bm k}_D}^\vee (\Delta_{{\bm i}_C}(\lambda))$ hold.
\end{thm}

\begin{proof}
We consider the Kashiwara embedding $\Psi_{\bm j} \colon \mathcal{B}(\infty) \hookrightarrow \z^\infty _{\bm j}$ associated with an infinite sequence ${\bm j} = (\ldots, j_2, j_1)$ extending ${\bm i}_C$ as in Sect.\ \ref{s:crystal_bases}. 
As in the proofs of \cref{t:main_result_type_A} and \cref{c:Kogan_face_string_polytope}, it suffices to show that 
\[\Psi_{\bm j} (\mathcal{B}_{w^{-1}}(\infty)) = \bigcup_{D \in \mathscr{M}(w)} F_{{\bm k}_D}^\vee({\bm j}),\] 
where 
\[F_{{\bm k}_D}^\vee({\bm j}) \coloneqq \{(\ldots, 0, 0, a_N, \ldots, a_2, a_1) \in \Psi_{\bm j} (\mathcal{B}(\infty)) \mid (a_1, \ldots, a_N) \in F_{{\bm k}_D}^\vee (\mathcal{C}_{{\bm i}_C})\}.\]

We proceed by induction on $\ell(w)$.
If $\ell(w) = 0$, then the assertion is obvious. 
If $\ell(w) = 1$, then we have $w = s_i$ for some $i \in I$.
By the definition of the crystal $\z^\infty _{\bm j}$, the image $\Psi_{\bm j} (\mathcal{B}_{s_i}(\infty)) = \{\tilde{f}_i^t \Psi_{\bm j} (b_\infty) \mid t \in \z_{\geq 0}\}$ coincides with the set of $(\ldots, a_l, \ldots, a_2, a_1) \in \z^\infty_{\bm j}$ such that $a_k = 0$ for all $k \neq \min\{l \in \z_{\geq 1} \mid j_l = i\}$.
Hence the assertion holds by \cref{ex:length_one_type_C}.

Assume that $\ell \coloneqq \ell(w) \geq 2$, and write ${\bm k}_{D(w)}^\prime = (k_1, \ldots, k_\ell)$. 
We set $i \coloneqq i_{k_\ell}$ and $w^\prime \coloneqq w s_i$. 
By the induction hypothesis, it follows that 
\begin{equation}\label{eq:main_type_C_induction_hypothesis}
\begin{aligned}
\Psi_{\bm j} (\mathcal{B}_{(w^\prime)^{-1}}(\infty)) = \bigcup_{D \in \mathscr{M}(w^\prime)} F_{{\bm k}_D}^\vee({\bm j}),
\end{aligned}
\end{equation}
and hence that 
\begin{equation}\label{eq:main_type_C_induction_hypothesis_cone}
\begin{aligned}
\mathcal{C}_{{\bm i}_C}(X_{w^\prime}) = \bigcup_{D \in \mathscr{M}(w^\prime)} F_{{\bm k}_D}^\vee(\mathcal{C}_{{\bm i}_C}).
\end{aligned}
\end{equation}
By the argument of the proof of \cref{t:main_result_type_A}, it suffices to show the following properties of $D \in \mathscr{M}(w^\prime)$ instead of Lemmas \ref{l:shape_of_reduced_pipe} and \ref{l:inverse_ladder_move}:
\begin{enumerate}
\item[{($P_1$)}] if an element $(p_k, q_k) \in D$ with $q_k \in \{n-i+1, n+i-1\}$ satisfies the following conditions:
\begin{itemize}
\item $(p_k, q_k+1) \notin D$,
\item there exists $k < r_k \leq N$ such that $q_{r_k} \in \{n-i+1, n+i-1\}$, $(p_{r_k}, q_{r_k}) \notin D$, and $(p_r, q_r), (p_r, q_r + 1) \in D$ for all $k < r < r_k$ such that $q_r \in \{n-i+1, n+i-1\}$,
\end{itemize}
then we have $(p_{r_k}, q_{r_k} + 1) \notin D$;
\item[{($P_2$)}] there exist $D^\prime \in \mathscr{M}(w^\prime)$ and $(u_1, v_1), \ldots, (u_r, v_r) \in SY_n$ with $v_1, \ldots, v_r \in \{n-i+1, n+i-1\}$ such that $D = L_{u_r, v_r} \cdots L_{u_1, v_1} (D^\prime)$ and such that $D^\prime$ cannot be written as $L_{u, v}(D'')$ for some $D'' \in \mathcal{SPD}_n$ and $(u, v) \in SY_n$ with $v \in \{n-i+1, n+i-1\}$.
\end{enumerate}
Let us take ${\bm a} = (\ldots, 0, 0, a_N, \ldots, a_2, a_1) \in \Psi_{\bm j}(\mathcal{B}(\infty))$ such that if $(p, q) \in SY_n \setminus D$, then $a_{(p, q)}$ is sufficiently larger than $a_{(p, q-1)}$, where we write $a_{(p_k, q_k)} \coloneqq a_k$ for $1 \leq k \leq N$ and set $a_{(p, p-1)} \coloneqq 0$ when $p = q$.
Then we see that $(a_1, \ldots, a_N)$ is included in the relative interior of $F_{{\bm k}_D}^\vee (\mathcal{C}_{{\bm i}_C})$.
By \eqref{eq:main_type_C_induction_hypothesis}, there exists $b \in \mathcal{B}_{(w^\prime)^{-1}}(\infty)$ such that $\Psi_{\bm j}(b) = {\bm a}$.

We first prove the property ($P_1$). 
In this case, we assume that $a_{(p_t, q_t)}-a_{(p_t, q_t-1)}$ is sufficiently larger than $a_{(p_r, q_r)}-a_{(p_r, q_r-1)}$ for $(p_t, q_t), (p_r, q_r) \in SY_n \setminus D$ such that $r < t$.
Suppose for a contradiction that there exists $(p_k, q_k) \in D$ with the conditions in ($P_1$) such that $(p_{r_k}, q_{r_k} + 1) \in D$.
We may assume that $k$ is maximum in such $k$. 
By the definition of $\z^\infty _{\bm j}$, it follows that the coordinate $a_k = a_{(p_k, q_k)}$ changes by the action of $\tilde{f}^t_i$ on ${\bm a}$ for sufficiently large $t$. 
Let $t_0$ denote the minimum of $t \in \z_{\geq 1}$ such that the action of $\tilde{f}^t_i$ changes the coordinate $a_k$.
In a way similar to the proof of \cref{t:main_result_type_A}, we know how the operator $\tilde{f}^{t_0}_i$ acts on ${\bm a}$. 
We write $\tilde{f}^{t_0}_i{\bm a} \eqqcolon {\bm a}^\prime = (\ldots, 0, 0, a_N^\prime, \ldots, a_2^\prime, a_1^\prime)$.
Then it follows that $(a_1^\prime, \ldots, a_N^\prime)$ is included in the relative interior of $F_{{\bm k}_{\widetilde{D}}}^\vee (\mathcal{C}_{{\bm i}_C})$, where $\widetilde{D}$ is given by
\begin{align*}
\widetilde{D} \coloneqq&\ D \setminus \{(p_t, q_t) \mid q_t \in \{n-i+1, n+i-1\},\ k \leq t,\ (p_t, q_t+1) \notin D\}\\ 
&\cup \{(p_t, q_t+1) \mid L^\prime_{p_t, q_t}\ {\rm is\ defined\ for}\ D \setminus \{(p_t, q_t)\}\}.
\end{align*}
Since $\tilde{f}^{t_0}_i b \in \mathcal{B}_{w^{-1}}(\infty)$ by \cref{p:property_of_Demazure_infty}, it follows by \cref{c:union_of_faces_string_cone} that $F_{{\bm k}_{\widetilde{D}}}^\vee (\mathcal{C}_{{\bm i}_C}) \subseteq \mathcal{C}_{{\bm i}_C}(X_w)$. 
This implies in a way similar to the proof of \cref{l:opposite_Demazure_union_of_faces} that $F_{{\bm k}_{\widetilde{D}}}^\vee (\Delta_{{\bm i}_C}(\lambda)) \subseteq \Delta_{{\bm i}_C}(\lambda, X_w)$ for all $\lambda \in P_+$.
For $\lambda \in P_{++}$, this gives a contradiction as follows. 
Since $|\widetilde{D}| = |D| - 2$ and $D \in \mathscr{M}(w^\prime)$, we have 
\[\dim_\r F_{{\bm k}_{\widetilde{D}}}^\vee (\Delta_{{\bm i}_C}(\lambda)) = \dim_\r F_{{\bm k}_{D}}^\vee (\Delta_{{\bm i}_C}(\lambda)) + 2 = \ell(w^\prime) + 2 = \ell(w) + 1.\]
However, it holds that $\dim_\r \Delta_{{\bm i}_C}(\lambda, X_w) = \ell(w)$ since $|\Delta_{{\bm i}_C}(\lambda, X_w) \cap \z^N| = |\Phi_{{\bm i}_C}(\mathcal{B}_w(\lambda))| = |\mathcal{B}_w(\lambda)| = \dim_\c H^0(X_w, \mathcal{L}_\lambda)$, which is a contradiction. 
This proves the property ($P_1$).

We next prove the property ($P_2$). 
In this case, we assume that $a_{(p_r, q_r)}-a_{(p_r, q_r-1)}$ is sufficiently larger than $a_{(p_t, q_t)}-a_{(p_t, q_t-1)}$ for $(p_r, q_r), (p_t, q_t) \in SY_n \setminus D$ such that $r < t$.
Write $\tilde{e}^{\varepsilon_i(b)}_i {\bm a} \eqqcolon {\bm a}^\prime = (\ldots, 0, 0, a_N^\prime, \ldots, a_2^\prime, a_1^\prime)$.
As in the proof of the property ($P_1$), we see that $(a_1^\prime, \ldots, a_N^\prime)$ is included in the relative interior of $F_{{\bm k}_{D^\prime}}^\vee (\mathcal{C}_{{\bm i}_C})$, where $D^\prime$ is given by
\begin{align*}
D^\prime \coloneqq&\ D \setminus \{(p_t, q_t+1) \mid q_t \in \{n-i+1, n+i-1\},\ (p_t, q_t) \in SY_n \setminus D\}\\ 
&\cup \{(p_t, q_t) \mid q_t \in \{n-i+1, n+i-1\},\ D \setminus \{(p_t, q_t+1)\} = L_{p_t, q_t} D''\ {\rm for\ some}\ D''\}.
\end{align*}
Since $\tilde{e}^{\varepsilon_i(b)}_i b \in \mathcal{B}_{(w^\prime)^{-1}}(\infty)$ by \cref{p:property_of_Demazure_infty}, it follows by \eqref{eq:main_type_C_induction_hypothesis_cone} that $F_{{\bm k}_{D^\prime}}^\vee (\mathcal{C}_{{\bm i}_C}) \subseteq \mathcal{C}_{{\bm i}_C}(X_{w^\prime}) = \bigcup_{\widehat{D} \in \mathscr{M}(w^\prime)} F_{{\bm k}_{\widehat{D}}}^\vee(\mathcal{C}_{{\bm i}_C})$ and that there exists $\widehat{D} \in \mathscr{M}(w^\prime)$ such that $\widehat{D} \subseteq D^\prime$. 
Since we have $(n+i-1, n-i+2) \notin SY_n$, the property ($P_1$) implies that for all $(p_t, q_t) \in SY_n \setminus D$ such that $q_t \in \{n-i+1, n+i-1\}$ and such that $(p_t, q_t+1) \in D$, there exist $D'' \in \mathcal{SPD}_n$ and $(p_r, q_r) \in SY_n$ with $q_r \in \{n-i+1, n+i-1\}$ such that $D \setminus \{(p_r, q_r+1)\} = L_{p_r, q_r} D''$. 
This implies that $|D^\prime| = |D|$, and hence that $D^\prime = \widehat{D} \in \mathscr{M}(w^\prime)$. 
Since $D^\prime$ has the desired property, we deduce the property ($P_2$). 
This proves the theorem.
\end{proof}

\begin{rem}
There exists a similarity property between string parametrizations of type $B_n$ and of type $C_n$ (see \cite[Sect.\ 5]{Kas6} and \cite[Sect.\ 6]{Fuj}). 
Hence \cref{t:main_result_string_Demazure_type_C} is naturally extended to the case of type $B_n$.
\end{rem}

Since we have $|\mathcal{B}_w(\lambda)| = \dim_\c H^0(X_w, \mathcal{L}_\lambda)$ for $w \in W$ and $\lambda \in P_+$, we obtain the following by \cref{t:main_result_string_Demazure_type_C}.

\begin{cor}\label{c:string_Demazure_volume_type_C}
For $w \in W$ and $\lambda \in P_+$, the dimension $\dim_\c H^0(X_w, \mathcal{L}_\lambda)$ equals the cardinality of $\bigcup_{D \in \mathscr{M}(w)} F_{{\bm k}_D}^\vee (SGT(\lambda)) \cap \z^N$.
In particular, it holds for $\lambda \in P_{++}$ that
\[{\rm Vol}(X_w, \mathcal{L}_\lambda) = \sum_{D \in \mathscr{M}(w)} {\rm Vol}_{\ell(w)} (F_{{\bm k}_D}^\vee (SGT(\lambda))).\]
\end{cor}

The following is an immediate consequence of Theorems \ref{t:semi-toric_degenerations} (3) and \ref{t:main_result_string_Demazure_type_C}.

\begin{cor}
For $w \in W$ and $\lambda \in P_{++}$, the Schubert variety $X_w$ degenerates into the union $\bigcup_{D \in \mathscr{M}(w)} X(F_{{\bm k}_D}^\vee (\Delta_{{\bm i}_C}(\lambda)))$ of irreducible normal toric subvarieties $X(F_{{\bm k}_D}^\vee (\Delta_{{\bm i}_C}(\lambda)))$ of $X(\Delta_{{\bm i}_C}(\lambda))$ corresponding to the faces $F_{{\bm k}_D}^\vee (\Delta_{{\bm i}_C}(\lambda))$, $D \in \mathscr{M}(w)$.
\end{cor}

It is easy to see that 
\begin{equation}\label{eq:additivity_of_SGT}
\begin{aligned}
SGT(\lambda + \mu) = SGT(\lambda) + SGT(\mu)\quad{\rm for\ all}\ \lambda, \mu \in P_+,
\end{aligned}
\end{equation}
and that the symplectic Gelfand--Tsetlin polytopes $SGT(\lambda)$, $\lambda \in P_{++}$, have the same normal fan which we denote by $\Sigma_C$.
The corresponding toric variety $X(\Sigma_C)$ is singular in general.
However, the proof of \cite[Proposition 4.1]{NNU} can also be applied to the fan $\Sigma_C$, and we obtain the following.

\begin{prop}\label{p:small_resolution_type_C}
Every refinement of the fan $\Sigma_C$ into simplicial cones without adding new rays gives a small resolution of $X(\Sigma_C)$. 
\end{prop}

For $\lambda \in P_+$ and ${\bm \epsilon} = (\epsilon_2, \ldots, \epsilon_n, \epsilon_1^\prime, \ldots, \epsilon_n^\prime) \in \z^{2n-1}$ such that $0 = \epsilon_1^\prime \leq \epsilon_2 \leq \epsilon_2^\prime \leq \cdots \leq \epsilon_n \leq \epsilon_n^\prime$, we define an integral convex polytope $\widetilde{SGT}_{\bm \epsilon}(\lambda)$ to be the set of \eqref{eq:coordinate_type_C_string} satisfying the following inequalities:
\begin{align*}
&b_j ^{(i)} + \epsilon_i^\prime \geq a_j ^{(i)} \geq b_{j+1} ^{(i)} & &{\rm for}\ 1 \leq i \leq n\ {\rm and}\ 1 \leq j \leq n-i+1,\\
&a_j ^{(i)} + \epsilon_{i+1} \geq b_j ^{(i+1)} \geq a_{j+1} ^{(i)} & &{\rm for}\ 1 \leq i \leq n-1\ {\rm and}\ 1 \leq j \leq n-i.
\end{align*}
If $\epsilon_i$ and $\epsilon_j^\prime$ are all $0$, then we have $\widetilde{SGT}_{\bm \epsilon}(\lambda) = SGT(\lambda)$. 
We fix ${\bm \epsilon} = (\epsilon_2, \ldots, \epsilon_n, \epsilon_1^\prime, \ldots, \epsilon_n^\prime) \in \z^{2n-1}$ such that $0 = \epsilon_1^\prime < \epsilon_2 < \epsilon_2^\prime < \cdots < \epsilon_n < \epsilon_n^\prime$. 
Then the polytope $\widetilde{SGT}_{\bm \epsilon}(\lambda)$ is simple, and its normal fan gives an example of \cref{p:small_resolution_type_C}.
By \eqref{eq:additivity_of_SGT}, the map $P_+ \rightarrow \mathscr{P}$, $\lambda \mapsto SGT(\lambda)$, induces an injective group homomorphism $P \hookrightarrow G(\mathscr{P})$. 
We define $R_{SGT}, R_{\widetilde{SGT}}$, and $M_{\widetilde{SGT}, SGT}$ as in Sect.\ \ref{s:type_A}.
Then it follows by \cref{t:Borel_description_polytope} and \cref{r:Borel_description_over_Z} that the polytope ring $R_{SGT}$ is isomorphic to the cohomology ring $H^\ast(G/B; \z)$. 
We fix $\lambda \in P_{++}$.
As in Sect.\ \ref{s:type_A}, a symplectic Kogan face $F_{\bm k}^\vee(SGT(\lambda))$ (resp., a dual symplectic Kogan face $F_{\bm k}(SGT(\lambda))$) gives an element $[F_{\bm k}^\vee(SGT)]$ (resp., $[F_{\bm k}(SGT)]$) of $M_{\widetilde{SGT}, SGT}$.
In a way similar to the proof of \cite[Theorem 4.3]{KST}, we obtain the following by \cref{c:dual_Kogan_opposite_Demazure_type_C}. 

\begin{cor}\label{c:description_of_Schubert_class_symplectic_Kogan}
For $w \in W$, the cohomology class $[X^w] = [X_{w_0 w}] \in H^\ast(G/B; \z) \simeq R_{SGT}$ is described as follows:
\begin{align*}
[X^w] &= \sum_{{\bm k} \in R({\bm i}_C, w)} [F_{\bm k}(SGT)].
\end{align*}
\end{cor}

\begin{ex}
Let $n = 2$. 
Then we can compute the product $[X^{s_1}] \cdot [X^{s_2}]$ using \cref{c:description_of_Schubert_class_symplectic_Kogan} as follows:
\begin{align*}
[X^{s_1}] \cdot [X^{s_2}] &= ([F_{(1)}(SGT)] + [F_{(3)}(SGT)]) \cdot ([F_{(2)}(SGT)] + [F_{(4)}(SGT)])\\
&= [F_{(1, 2)}(SGT)] + [F_{(1, 4)}(SGT)] + [F_{(2, 3)}(SGT)] + [F_{(3, 4)}(SGT)]\\
&= [X^{s_1 s_2}] + [X^{s_2 s_1}].
\end{align*}
\end{ex}

Similarly, the following holds by \cref{c:string_Demazure_volume_type_C}. 

\begin{cor}\label{c:description_of_Schubert_class_dual_symplectic_Kogan}
For $w \in W$, the cohomology class $[X_w] = [X^{w_0 w}] \in H^\ast(G/B; \z) \simeq R_{SGT}$ is described as follows:
\begin{align*}
[X_w] &= \sum_{D \in \mathscr{M}(w)} [F_{{\bm k}_D}^\vee(SGT)].
\end{align*}
\end{cor}

For each ${\bm k}$, let $F_{\bm k}^\vee(\widetilde{SGT}_{\bm \epsilon}(\lambda))$ (resp., $F_{\bm k}(\widetilde{SGT}_{\bm \epsilon}(\lambda))$) denote the face of $\widetilde{SGT}_{\bm \epsilon}(\lambda)$ corresponding to the face $F_{\bm k}^\vee(SGT(\lambda))$ (resp., $F_{\bm k}(SGT(\lambda))$) of $SGT(\lambda)$. 
Then the faces $F_{\bm k}^\vee(\widetilde{SGT}_{\bm \epsilon}(\lambda))$ and $F_{{\bm k}^\prime}(\widetilde{SGT}_{\bm \epsilon}(\lambda))$ intersect transversally for each ${\bm k}, {\bm k}^\prime$ if their intersection is nonempty.
Hence the following is an immediate consequence of Corollaries \ref{c:description_of_Schubert_class_symplectic_Kogan} and \ref{c:description_of_Schubert_class_dual_symplectic_Kogan}.

\begin{cor}\label{c:description_Schubert_calculus_type_C}
For $v, w \in W$, the product $[X^v] \cdot [X^w] \in H^\ast(G/B; \z) \simeq R_{SGT}$ is described as follows:
\begin{align*}
[X^v] \cdot [X^w] &= \sum_{{\bm k} \in R({\bm i}_C, v)} \sum_{D \in \mathscr{M}(w_0 w)} [F_{\bm k}(\widetilde{SGT}_{\bm \epsilon}(\lambda)) \cap F_{{\bm k}_D}^\vee(\widetilde{SGT}_{\bm \epsilon}(\lambda))].
\end{align*}
\end{cor}

\bibliographystyle{jplain} 
\def\cprime{$'$} 

\end{document}